\documentclass[final,reqno]{amsart}

\usepackage[UKenglish]{babel}
\usepackage{amssymb}
\usepackage{IEEEtrantools}
\usepackage{tikz-cd}
\usepackage{mathtools}
\usepackage{url}

\newcommand{\field}[1]{\mathbb{#1}} 
\newcommand{\curl}[1]{\mathcal{#1}} 
\newcommand{\ideal}[1]{\mathfrak{#1}} 

\newcommand{\Z}{\field{Z}} 
\newcommand{\Q}{\field{Q}} 
\newcommand{\Proj}{\field{P}} 
\newcommand{\F}{\field{F}} 

\DeclareMathOperator{\residue}{Res} 
\DeclareMathOperator{\order}{ord} 
\DeclareMathOperator{\resultant}{Res} 
\DeclareMathOperator{\rad}{Rad} 
\DeclareMathOperator{\End}{End} 
\DeclareMathOperator{\GL}{GL} 
\DeclareMathOperator{\Brauer}{Br} 
\DeclareMathOperator{\frob}{Frob} 
\DeclareMathOperator{\tr}{Tr} 

\newcommand{\inset}[2]{\in \left\{ #1 , \ldots , #2 \right\} } 

\theoremstyle{plain}                    
\newtheorem{theorem}{Theorem}[section] 
\newtheorem{lemma}[theorem]{Lemma} 
\newtheorem{proposition}[theorem]{Proposition} 
\newtheorem{corollary}[theorem]{Corollary} 

\theoremstyle{definition}

\theoremstyle{remark}
\newtheorem*{remark}{Remark} 

\begin{document}
\title{On the sum of fourth powers in arithmetic progression}
\author[J.M. van Langen]{Joey M. van Langen}
\address{Department of Mathematics, {V}rije {U}niversiteit {A}msterdam, De Boelelaan 1111, 1081 HV Amsterdam, The Netherlands}
\email{j.m.van.langen@vu.nl}
\thanks{Supported by NWO Vidi grant 639.032.613}

\begin{abstract}
  We prove that the equation~${ (x - y)^4 + x^4 + (x + y)^4 = z^n }$
  has no integer solutions~${ x, y, z}$ with~${ \gcd(x, y) = 1 }$ for
  all integers~${ n > 1 }$. We mainly use a modular approach with two
  Frey~${ \Q }$-curves defined over the field~${ \Q( \sqrt{30} ) }$.
\end{abstract}


\maketitle

\section{Introduction}
\label{sec:intro}
In this paper we will study an equation of the form
\begin{equation}
  \label{eq:threetermform}
  (x-y)^k + x^k + (x+y)^k = z^n, \quad x, y, z \in \Z,
\end{equation}
i.e. the sum of three~${ k }$-th powers in arithmetic progression
being a perfect power. Such equations have been intensively studied in
the case~${ y = 1 }$, i.e. consecutive~${ k }$-th powers. The earliest
results in that case were already formulated by Euler in the
case~${ k = n = 3 }$. Zhang \cite{Zhang} gave a complete solution for
consecutive integers for~${ k = 2, 3, 4 }$ and this was extended by
Bennett, Patel and Siksek \cite{BennettPatelSiksek}
for~${ k = 5, 6 }$. In both cases the modular method was used with
Frey curves defined over the rationals.

Also the more general case of equation~\eqref{eq:threetermform} has
been studied before. For the case~${ k = 2 }$ and~${ \gcd(x, z) = 1 }$
Koutsianas and Patel \cite{KoutsianasPatel} used prime divisors of
Lehmer sequences to determine all solutions
when~${ 1 \le y \le 5000 }$. Koutsianas \cite{Koutsianas} further
studied this case when~${ y }$ is a prime power~${ p^m }$ for specific
prime numbers~${ p }$. The case~${ k = 3 }$ was partially solved by
Arg\'aez-Garc{\'\i}a and Patel \cite{ArgaezGarciaPatel} giving all
solutions in case~${ 1 \le y \le 10^6 }$ using different techniques
including the modular method for some Frey curves over the
rationals. In~\cite{KoutsianasPatel} and \cite{ArgaezGarciaPatel} the
bounds on~${ y }$ are merely for computational purposes, whilst the
techniques would generalize to larger bounds.

Variants of equation~\eqref{eq:threetermform} with more terms on the
left-hand side have also been studied. Recent results include those by
Patel and Siksek~\cite{PatelSiksek} and Patel~\cite{Patel}.

In this paper we look at equation~\eqref{eq:threetermform} for the
case~${ k = 4 }$. Zhang \cite{Zhang2} proved a partial result in this
case. By considering~${ y }$ as a parameter and using the modular
method with Frey curves over~${ \Q }$, he managed to prove the
non-existence of solutions for certain families of values for~${ y
}$. Although his approach could be pushed to include more families of
values for~${ y }$ it appears this method can not be generalized to
treat all values of~${ y }$ simultaneously.

We will give a complete solution for the case~${ k = 4 }$
(where~${ \gcd(x, y) = 1 }$ as always). Essential in the proof is the
construction of two Frey ${ \Q }$-curves defined
over~${ \Q( \sqrt{ 30 } ) }$. Using the modular method on these curves
overcomes the limitations in \cite{Zhang2}, allowing us to prove the
following main result.

\begin{theorem}
  \label{thm:main}
  The sum of three coprime fourth powers in arithmetic progression is
  not a perfect power, i.e. the equation
  \begin{equation}
    \label{eq:main}
    (x - y)^4 + x^4 + (x + y)^4 = z^l
  \end{equation}
  has no solutions~${ x, y, z \in \Z }$ with~${ \gcd (x, y) = 1 }$
  for integers~${ l > 1 }$.
\end{theorem}
Note that any solution to equation~\eqref{eq:main} gives rise to a
solution for~${ l }$ a prime number. For our proof it thus suffices to
prove Theorem~\ref{thm:main} for~${ l }$ prime as we shall do
throughout this paper.

As mentioned, the construction of two Frey curves over the
field~${ \Q( \sqrt{30} ) }$ will be essential in the proof. The
construction of these curves can be found in
Section~\ref{sec:FreyCurves}.

Since~${ \Q( \sqrt{30} ) }$ is a real quadratic field the most direct
approach to apply the modular method is to use Hilbert modularity of
curves defined over real quadratic fields. We will perform the initial
steps to this approach in Section~\ref{sec:hilbert}. However we will
also argue that the computation of the corresponding spaces of Hilbert
modular forms is out of reach for the current computational power,
making this approach unfeasible.

Instead we will use that the curves in this paper are by construction
also~${ \Q }$-curves for which a separate modularity result is known
\cite{Ribet}. A~${ \Q }$-curve approach to solving Diophantine
equations has already been used in articles such as the ones by
Ellenberg \cite{Ellenberg}, Dieulefait and Freitas
\cite{DieulefaitFreitas}, Dieulefait and Urroz \cite{DieulefaitUrroz},
Chen \cite{Chen1,Chen}, Bennett and Chen \cite{BennettChen}, and
Bennett, Chen, Dahmen and Yazdani \cite{BennettChenDahmenYazdani}. We
will discuss this approach in Section~\ref{sec:Qcurves}.

As in the mentioned articles we will follow \cite{Quer} for general
results about~${ \Q }$-curves. The main differences lie in that the
restrictions of scalars of our curves are not abelian varieties
of~${ \GL_2 }$-type themselves, but will decompose as a product of
such varieties. This also happens in \cite{DieulefaitFreitas} and
\cite{Chen1}. In the first this issue is dealt with by studying the
relation between the corresponding Galois representations. We will
rather study the relation between the corresponding newforms as was
done in \cite{Chen1} and will be more specific about computing the
character that defines this relation.

Another difference is in the way we compute the elliptic curve of
which one should take the restriction of scalars. Most mentioned
articles simply refer to \cite{Quer2} to prove the existence of a
twist of the original curve that will suffice and perform a small
search to find this twist. In \cite{BennettChen} a more direct
approach is given in case one can find a
map~${ \alpha: G_\Q \to \curl{O}_{K}^{*} }$ with a certain
coboundary. We show that also in our case we can find such a map if
the sought twist exists.

Furthermore the approach we took should generalize to other
Frey~${ \Q }$-curves and is mostly algorithmic. Therefore the author
has written code~\cite{code} for SAGE~\cite{sage} that automates the
generic parts of this approach. This code includes SAGE code to work
with Frey curves, ${\Q}$-curves and Frey~${ \Q }$-curves as well as
the associated newforms. A reasonable effort has been made to make the
code work for general such curves and provide sufficient
documentation. Furthermore the file \texttt{calculations.rst} in
\cite{code} provides a structured overview of all the computations
done for this paper interjected with explanatory notes. This document
describes all intermediate steps needed for the calculations in this
paper such that one can easily reproduce the results. Furthermore the
computations in this file can be verified automatically using SAGE's
automated doctest system. Some computations in this file also make use
of MAGMA~\cite{magma}. The code~\cite{code} also provides support for
working with MAGMA when computing newforms to decrease computation
time.

Since the modular method approach in this setting only works for
primes~${ l > 5 }$, Section~\ref{sec:small_l} is dedicated to proving
the cases for small~${ l }$. The case~${ l = 2 }$ follows immediately
from a local obstruction, whereas the cases~${ l = 3 }$
and~${ l = 5 }$ require the computation of some points on
hyperelliptic curves to prove the non-existence of solutions.

Section~\ref{sec:pre} introduces some preliminary results about
equation~\eqref{eq:main} and introduces some notation that will be
used throughout the paper.

\subsubsection*{Acknowledgments}
Many thanks go to my supervisor Sander Dahmen, who initiated this
project, suggested the two Frey curves, and helped at many points
along the way. Thanks to the referee for their helpful comments to
improve the exposition.


\section{Preliminaries}
\label{sec:pre}

In this section we will prove some general results about integer
solutions to equation~\eqref{eq:main} with~${ \gcd(x, y) = 1
}$. Throughout the paper~${ (a, b, c) }$ will denote an arbitrary such
solution. Note that we have
\begin{equation}
  \label{eq:main2}
  c^l = (a - b)^4 + a^4 + (a + b)^4 = 3 \, a^4 + 12 \, a^2 b^2 + 2 b^4,
\end{equation}
which leads to the following result.
\begin{proposition}
  \label{thm:c_not_div}
  The integer~${ c }$ is not divisible by~${ 2 }$,~${ 3 }$ or~${ 5 }$,
  hence~${ a }$ is odd and~${ b }$ is not divisible by~${ 3 }$.
\end{proposition}
\begin{proof}
  This follows immediately by considering equation~\eqref{eq:main2}
  modulo~${ 4 }$,~${ 9 }$ and~${ 5 }$.
\end{proof}

Let~${ f(x, y) }$ be the left-hand side of
equation~\eqref{eq:main}. The most general factorization of~${ f }$ is
obtained by factoring over the splitting field~${ L }$
of~${ f(x, 1) }$.  Since~${ f(x, 1) }$ is irreducible, the
polynomial~${ f(x, y) }$ factors as a product of a constant and Galois
conjugates of
\begin{equation}
  \label{eq:hdef}
  h( x, y ) = x + v y,
\end{equation}
where~${ v }$ is a root of~${ f(x, 1) }$. To be precise we have
\begin{equation}
  \label{eq:factorL}
  f(x, y) = 3 \, (x + v y) \, (x - v y) \, (x + \sqrt{(-v^2 - 4)} \, y ) \, (x
  - \sqrt{(-v^2 - 4)} \, y). 
\end{equation}
We can say a lot about the factor~${ h(a, b) }$ and its Galois conjugates.

\begin{lemma}
  \label{thm:factor_coprime}
  The distinct Galois conjugates of~${ h(a, b) }$ are coprime outside
  primes above~${ 3 }$. Furthermore, the valuation of~${ h(a,b) }$ at
  all primes above~${ 2 }$ and~${ 5 }$ is zero and its valuation at
  the unique prime~${ \ideal{p}_3 }$ in~${ L }$ above~${ 3 }$
  is~${ -1 }$.
\end{lemma}
\begin{proof}
  Since any field that contains~${ h(a, b) }$ and its Galois
  conjugates contains~${ L }$ we can safely do all computations
  in~${ L }$. Note that~${ a }$ and~${ b }$ are integers and
  that~${ v }$ is only not integral at the unique
  prime~${ \ideal{p}_3 }$ above~${ 3 }$, so the only prime at
  which~${ h(a, b) }$ and its Galois conjugates are not integral
  is~${ \ideal{p}_3 }$. For two distinct Galois
  conjugates~${ \prescript{\sigma}{}{(h(a, b))} }$
  and~${ \prescript{\tau}{}{(h(a,b))} }$ their difference is equal
  to~${ \left( \prescript{\sigma}{}{v} - \prescript{\tau}{}{v} \right)
    b }$. Any prime dividing both can not divide~${ b }$ as it then
  also divides~${ a }$ contradicting their coprimality. Therefore the
  only common primes are in the
  differences~${ \prescript{\sigma}{}{v} - \prescript{\tau}{}{v}
  }$. Using SAGE~\cite{sage} we can determine that the only primes
  dividing these difference are those above~${ 2 }$,~${ 3 }$
  and~${ 5 }$, hence we arrive at the first conclusion.

  For the second statement, we note that~${ a }$ and~${ b }$ are
  integral and~${ v }$ only has negative valuation
  at~${ \ideal{p}_3 }$ with~${ \order_{\ideal{p}_3}(v) = -1 }$. This
  implies that the valuation of~${ h(a,b) }$ and its conjugates is at
  least~${ 0 }$ at all primes above~${ 2 }$ and~${ 5 }$ and at
  least~${ -1 }$ at~${ \ideal{p}_3 }$. Applying this information to
  equation~\eqref{eq:factorL} and using that~${ c }$ has
  valuation~${ 0 }$ at all these primes by
  Proposition~\ref{thm:c_not_div}, the second result immediately
  follows.
\end{proof}

Lemma~\ref{thm:factor_coprime} and equation~\eqref{eq:factorL} tell us that
\begin{equation*}
  (h(a,b)) = \ideal{p}_3^{-1} I^l
\end{equation*}
for some integral ideal~${ I }$ of~${ L }$. We will need this general
result to solve the case~${ l = 5 }$.

For the other cases, we can limit ourselves to the
subfield~${ K = \Q(\sqrt{30}) }$ of~${ L }$. In this case we have two
factors
\begin{IEEEeqnarray}{rCl}
  \label{eq:gdef}
  g_1(x, y) & := & x^2 + \left(2 + \frac{1}{3}\sqrt{30}\right)y^2 \\
  g_2(x, y) & := & x^2 + \left(2 - \frac{1}{3}\sqrt{30}\right)y^2,
\end{IEEEeqnarray}
and the factorization is
\begin{equation}
  \label{eq:factorK}
  z^l = f(x, y) = 3 \, g_1(x, y) \, g_2(x, y).
\end{equation}
Note that~${ g_1 }$ and~${ g_2 }$ are both the product of two Galois
conjugates of~${ h }$ and since these are all distinct we can conclude
that~${ g_1(a,b) }$ and~${ g_2(a,b) }$ are coprime outside primes
above~${ 3 }$. Also the result about valuations carries over, so both
have valuation~${ 0 }$ at primes above~${ 2 }$ and~${ 5 }$ and since
the unique prime~${ \ideal{q}_3 }$ of~${ K }$ above~${ 3 }$ factors
as~${ \ideal{p}_3^2 }$ in~${ L }$ they have valuation~${ -1 }$
at~${ \ideal{q}_3 }$. Furthermore we find that
\begin{IEEEeqnarray*}{rCl}
  (g_1(a,b)) & = & \ideal{q}_3^{-1} I_1^l \\
  (g_2(a,b)) & = & \ideal{q}_3^{-1} I_2^l,
\end{IEEEeqnarray*}
with~${ I_1 }$ and~${ I_2 }$ coprime integral ideals of~${ K }$.

Throughout this paper~${ K }$ and~${ L }$ will be the same as in this
section, as will~${ g_1 }$,~${ g_2 }$ and~${ h }$.

\section{Cases for small ${ l }$}
\label{sec:small_l}
In this section we will solve equation~\eqref{eq:main} for small prime
exponents~${ l }$ as the modular method used in the next sections only
works for~${ l > 5 }$. All small cases have a slightly different
approach.

\subsection{Case~${ l = 2 }$}
In case~${ l = 2 }$ equation~\eqref{eq:main} has a local obstruction
at~${ 3 }$. This can be seen by considering equation~\eqref{eq:main2}
modulo 9, or by considering the equation modulo~${ 3 }$ and using that
${ 3 \nmid c }$ from Proposition~\ref{thm:c_not_div}. This proves the
non-existence of solutions in this case. A similar obstruction can be
found modulo~${ 5 }$.

\subsection{Case ${l = 3}$}
\label{sec:l3}
Suppose~${ (a, b, c) }$ is a solution to equation~\eqref{eq:main}
for~${ l = 3 }$ and assume that~${ \gcd(a, b) = 1 }$. From
Section~\ref{sec:pre} we know that
\begin{equation*}
  \left( g_1(a, b) \right) = \ideal{q}_3^{-1} I_1^3 = \ideal{q}_3^{-4} \left( \ideal{q}_3 I_1 \right)^3
\end{equation*}
as fractional ideals in~${ K }$. Since~${ K }$ has class
number~${ 2 }$ and~${ \ideal{q}_3^{-4} = \left( \frac{1}{9} \right) }$
we find that~${ \ideal{q}_3 I_1 }$ must be a principal ideal. Hence we
conclude that
\begin{equation*}
  g_1(a, b) = \frac{1}{9} u \gamma^3,
\end{equation*}
for~${ u \in \curl{O}_{K}^* }$ and~${ \gamma \in \ideal{q}_3 }$. Note
that~${ \curl{O}_K^* }$ is generated by~${ u_0 = (-1)^3 }$ and an
element~${ u_1 }$ of infinite order, hence we can
take~${ u = u_1^{j} }$ for~${ j = 0, 1, 2 }$.
Since~${ \{ 3, 6 + \sqrt{30} \} }$ is an integral basis of
of~${ \ideal{q}_3 }$ we can parametrize~${ \gamma }$ with integral
parameters~${ s }$ and~${ t }$ to get that
\begin{IEEEeqnarray}{rCl}
  \gamma & = & 3 s + \left( 6 + \sqrt{30} \right) t =  3 g_1(\sqrt{s}, \sqrt{t}) \nonumber \\
  a^2 & = & F_{3, j}(s, t) \label{eq:param3} \\
  b^2 & = & G_{3, j}(s, t) \nonumber \\
  c & = & 3 \, s^2 + 12 \, s t + 2 t^2 \nonumber
\end{IEEEeqnarray}
for~${ F_{3, j}(s, t) }$ and~${ G_3(s, t) }$ some homogeneous
polynomials over~${ \Q }$ of degree~${ 3 }$.

Note that~${ t = 0 }$ corresponds to solutions in which~${ c }$ would
be divisible by~${ 3 }$. So by Proposition~\ref{thm:c_not_div} we find
that~${ t \neq 0 }$. By multiplying the middle two equations in
\eqref{eq:param3} and dividing by~${ t^6 }$ we find hyperelliptic
curves~${ C_j }$ in the variables~${ X = \frac{s}{t} }$
and~${ Y = \frac{ a b }{ t^3 } }$. Explicitly these curves are given
by
\begin{IEEEeqnarray*}{rCrCl}
  C_0 & : & Y^2 & = & 27 \,X^{5} + 108 \,X^{4} + 84 \,X^{3} - 288 \,X^{2} - 564 \,X - 368 \\
  C_1 & : & Y^2 & = & -1\,242\,X^6 - 1\,269\,X^5 - 432\,X^4 + 84\,X^3 + 72\,X^2 + 12\,X \\
  C_2 & : & Y^2 & = & -599\,940\,X^6 - 627\,237\,X^5 - 273\,132\,X^4 - 63\,276\,X^3 \\
  &&&& \quad - 8\,208\,X^2 - 564\,X - 16.
\end{IEEEeqnarray*}
As~${ t \neq 0 }$ every solution~${ (a, b, c) }$ corresponds to a
point on such a curve. Using MAGMA~\cite{magma} we see that the
curve~${ C_2 }$ has no solution in~${ \Q_3 }$, hence none in~${ \Q
}$. Also using MAGMA we can compute that the Jacobian of~${ C_0 }$ has
only two-torsion points as rational points. Note that such points
correspond to factors of the defining polynomial, i.e.
of~${ F_{3, 0}(s, t) G_{3, 0}(s, t) }$. Since the only linear factor
in~${ F_{3, 0} G_{3, 0} }$ is~${ t }$, the only rational point
on~${ C_0 }$ corresponds to the case~${ t = 0 }$ which we already
excluded from corresponding to a solution.

The curve~${ C_1 }$ has no local obstruction. Furthermore the rank of
its Jacobian is bounded above by 1 and its L-function suggests the
rank is 1. However no point of infinite order on the Jacobian can be
found within a small bound. We shall therefore apply a different
approach.

For the case~${ j = 1 }$ the equations in \eqref{eq:param3} become
explicitly
\begin{IEEEeqnarray*}{rCl}
  a^2 & = & - 3 \, s \left( 23 \, s^2 + 12 \, s t + 2 \, t^2 \right) \\
  b^2 & = & 18 \, s^3 + 9 \, s^2 t - 2 \, t^3 \\
  c & = & 3 \, s^2 + 12 \, s t + 2 \, t^2.
\end{IEEEeqnarray*}
Note that in the first equation~${ 23 \, s^2 + 12 \, s t + 2 \, t^2 }$
is congruent to~${ c }$ modulo~${ 20 }$. By
Proposition~\ref{thm:c_not_div} this
implies~${ 23 \, s^2 + 12 \, s t + 2 \, t^2 }$ is not divisible
by~${ 2 }$ or~${ 5 }$. Note that~${ s }$ and~${ t }$ must be coprime
since~${ a }$ and~${ b }$ are coprime, hence~${ s }$
and~${ 23 \, s^2 + 12 \, s t + 2 \, t^2 }$ must be coprime
outside~${ 2 }$. Since~${ 2 }$ does not
divide~${ 23 \, s^2 + 12 \, s t + 2 \, t^2 }$ the two must be coprime
and we find that
\begin{IEEEeqnarray*}{rCl}
  a & = & 3 \, a_1 \, a_2 \\
  s & = & (-1)^{e_1} \, 3^{e_2} \, a_1^2 \\
  23 \, s^2 + 12 \, s t + 2 \, t^2 & = & (-1)^{1 - e_1} \, 3^{1 - e_2} \, a_2^2,
\end{IEEEeqnarray*}
with~${ e_1, e_2 \in \{ 0, 1 \} }$. Now note that
\begin{equation*}
  23 \, s^2 + 12 \, s t + 2 \, t^2
  = 2 \left( \beta s + t \right) \left( \overline{\beta} s + t \right)
  = 2 N_{\Q}^{\Q(\sqrt{-10})} \left( \beta s + t \right),
\end{equation*}
where~${ \beta = 3 + \sqrt{10} / 2 }$,~${ \overline{\beta} }$ is its
Galois conjugate and~${ N_{\Q}^{\Q( \sqrt{-10} ) } }$ is the norm of
the field~${ \Q( \sqrt{-10} ) }$. Since~${ \Q( \sqrt{-10} ) }$ is an
imaginary field its norm is positive, hence~${ e_1 = 1 }$. Furthermore
the unique prime above~${ 3 }$ in~${ \Q( \sqrt{-10} ) }$ has
norm~${ 9 }$, so~${ e_2 = 1 }$. We thus have that
\begin{equation*}
  2 \left( \beta s + t \right) \left( \overline{\beta} s + t \right) = a_2^2.
\end{equation*}
Note that the factors~${ \beta s + t }$
and~${ \overline{\beta} s + t }$ are coprime outside primes dividing
the difference~${ \beta - \overline{\beta} = \sqrt{-10} }$.
Since~${ a_2 }$ is an integer not divisible by~${ 2 }$ and~${ 5 }$
which ramify in~${ \Q( \sqrt{-10} ) }$ these factors must have
valuation~${ -1 }$ at the unique prime~${ \ideal{p}_2 }$ above 2 and
valuation~${ 0 }$ at the unique prime above 5. This implies that
\begin{equation*}
  \left( \beta s + t \right) = \ideal{p}_2^{-1} I^2
\end{equation*}
for some ideal~${ I }$ of~${ \curl{O}_{\Q(\sqrt{-10})} }$. Since the
class number of~${ \Q(\sqrt{-10}) }$ is~${ 2 }$ this would imply
that~${ \ideal{p}_2 }$ is principal, which is not the case. Therefore
no solution can correspond to the case~${ j = 1 }$ and thereby no
solution to equation~\eqref{eq:main} with~${ \gcd( x, y ) = 1 }$
exists for~${ l = 3 }$.

\begin{remark}
  Geometrically, equations~\eqref{eq:param3} define a curve with a
  degree 4 map to~${ \Proj^1 }$. The hyperelliptic curves we
  constructed in these sections are quotients of these curves through
  which this degree 4 map factors as two degree 2 maps. The only other
  geometric quotients with this same property are defined by taking
  the equation for~${ a^2 }$ and setting~${ t = 1 }$ or taking the
  equation for~${ b^2 }$ and setting~${ t = 1 }$. The quotient we
  considered is the only quotient for which rational points correspond
  to rational solutions of the corresponding equations, making this
  the natural choice. The same will be true for the
  equations~\eqref{eq:param5} in the next section.
\end{remark}


\subsection{Case ${ l = 5 }$}
\label{sec:l5}

Now suppose we have a solution~${ (a, b, c) }$ to
equation~\eqref{eq:main} with~${ l = 5 }$ and~${ \gcd(a, b) =
  1}$. Recall from Section~\ref{sec:pre} that
\begin{equation*}
  (h(a,b)) = \ideal{p}_3^{-1} I^5
  = \ideal{p}_3^4 \left(\ideal{p}_3^{-1} I\right)^5
\end{equation*}
as fractional ideals in~${ L }$.
Since~${ \ideal{p}_3^4 = (3) }$ and~${ 5 }$ does not divide the order
of the class group of~${ L }$ we find that~${ \ideal{p}_3^{-1} I }$
must be a principal ideal, hence
\begin{equation*}
  h(a,b) = 3 \, u \, \gamma^5
\end{equation*}
for some unit~${ u \in \curl{O}_L^* }$ and~${ \gamma \in L
}$. Furthermore the valuation of~${ \gamma }$ is only negative
at~${ \ideal{p}_3 }$ where it is~${ -1 }$,
hence~${ \gamma \in \frac{1}{3} \curl{O}_L }$. Note that in these
arguments we may also replace~${ L }$ with~${ \Q(v) }$. The
field~${ \Q(v) }$ is a number field of degree~${ 4 }$ and
hence~${ \curl{O}_{\Q(v)} }$ can be parameterized by four integer
coefficients. Using this description we obtain for each choice
of~${ u }$ a parameterization of~${ a }$ and~${ b }$ together with two
equations in the four indeterminates. Since~${ \curl{O}_{\Q(v)}^* }$
is generated by~${ u_0 = (-1)^5 }$ and an element~${ u_1 }$ of
infinite order it is sufficient to consider~${ u = u_1^i }$
for~${ i \inset{0}{4} }$.  Considering the equations obtained for each
of these~${ i }$ modulo 5, we see that only two of them can
parameterize coprime~${ a }$ and~${ b }$, leaving only the
cases~${ i = 0, 4 }$.

Without loss of generality we may assume that~${ g_1 }$ is the product
of~${ h }$ with~${ \prescript{\sigma}{}{h} }$, where~${ \sigma }$ is the
automorphism on~${ \Q(v) }$ mapping~${ v }$ to~${ -v }$. This implies
that we have
\begin{equation*}
  g_1(a, b) = 9 \left(u \prescript{\sigma}{}{u}\right)
  \left(\gamma \prescript{\sigma}{}{\gamma}\right)^5.
\end{equation*}
Since~${ \Q(v^2) = \Q(\sqrt{30}) = K }$ we find
that~${ u' := u \prescript{\sigma}{}{u} \in \curl{O}_K^* }$
and~${ \gamma' := \gamma \prescript{\sigma}{}{\gamma} \in K }$. Note
that again~${ \gamma' }$ is only not integral at~${ \ideal{q}_3 }$ and
furthermore~${ \order_{\ideal{q}_3}(\gamma) = -1 }$, so we have
that~${ \gamma' \in \ideal{q}_3^{-1} = \left( \frac{1}{3} \right)
  \ideal{q}_3 }$. From the case~${ l = 3 }$ we know
that~${ \ideal{q}_3 }$ has an integral basis formed by~${ 3 }$ times
the coefficients of~${ g_1 }$, hence~${ \ideal{q}_3^{-1} }$ has an
integral basis formed by the coefficients of~${ g_1 }$. In particular
it has an integral basis of the
form~${ \left\{ 1 , \frac{\sqrt{30}}{3} \right\} }$ which gives us a
parameterization of the form
\begin{IEEEeqnarray}{rCl}
  \gamma' & = & s + \frac{\sqrt{30}}{3} t \nonumber \\
  a^2 & = & F_{5, u'} (s, t) \nonumber \\
  b^2 & = & G_{5, u'} (s, t), \label{eq:param5} \\
  c & = & 3 \, s^{2} - 10 \, t^{2} \nonumber
\end{IEEEeqnarray}
for each choice of unit~${ u' }$. Here~${ F_{5, u'} }$ and~${ G_{5, u'}
}$ are homogeneous polynomials over~${ \Q }$ of degree~${ 5 }$.

The only remaining cases are~${ u' = 1 \prescript{\sigma}{}{1} = 1 }$
and~${ u' = u_1^4 \prescript{\sigma}{}{u_1^4} = u_1^8 }$. As in the
case~${ l = 3 }$ we can see that~${ t \ne 0 }$ and hence we can
construct hyperelliptic curves by multiplying the equations
for~${ a^2 }$ and~${ b^2 }$ and dividing the result by~${ t^{10}
}$. These give us the hyperelliptic curves
\begin{IEEEeqnarray*}{rCrCl}
  C_1 & : & Y^2
  & = & 405 \,X^{9}
  - 4\,050 \,X^{8}
  + 16\,200 \,X^{7}
  - 54\,000 \,X^{6}
  + 113\,400 \,X^{5} \\ &&&& \quad
  - 198\,000 \,X^{4} 
  + 180\,000 \,X^{3}
  - 120\,000 \,X^{2}
  + 50\,000 \,X
  - 20\,000 \\
  C_{u_1^8} & : & Y^2
  & = & -3\,083\,903\,014\,930\,297\,409\,520 \,X^{10} \\ &&&& \quad
  - 56\,304\,108\,214\,517\,165\,808\,555 \,X^{9} \\ &&&& \quad
  - 462\,585\,452\,239\,544\,611\,432\,050 \,X^{8} \\ &&&& \quad
  - 2\,252\,164\,328\,580\,686\,632\,342\,200 \,X^{7} \\ &&&& \quad
  - 7\,195\,773\,701\,504\,027\,288\,934\,000 \,X^{6} \\ &&&& \quad
  - 15\,765\,150\,300\,064\,806\,426\,395\,400 \,X^{5} \\ &&&& \quad
  - 23\,985\,912\,338\,346\,757\,629\,798\,000 \,X^{4} \\ &&&& \quad
  - 25\,024\,048\,095\,340\,962\,581\,580\,000 \,X^{3} \\ &&&& \quad
  - 17\,132\,794\,527\,390\,541\,164\,120\,000 \,X^{2} \\ &&&& \quad
  - 6\,951\,124\,470\,928\,045\,161\,550\,000 \,X \\ &&&& \quad
  - 1\,269\,095\,890\,917\,817\,864\,020\,000
\end{IEEEeqnarray*}
in the variables~${ X = \frac{s}{t} }$ and~${ Y = \frac{a b}{t^5}
}$. Studying the Jacobians of these curves in MAGMA~\cite{magma} we
find that both Jacobians only contain two-torsion points. Since the
polynomials~${ F_{5, 1} G_{5, 1} }$
and~${ F_{5, u_1^8} G_{5, u_1^8} }$ only contain one linear factor we
conclude as in the case~${ l = 3 }$ that both curves only have one
rational point. These points are a point at infinity for~${ C_1 }$
corresponding to~${ t = 0 }$, and the point~${ (- \frac{42}{23}, 0) }$
on~${ C_2 }$. Note that these rational points correspond to values
for~${ s }$ and~${ t }$ for which either~${ 2 \mid c }$
or~${ 3 \mid c }$ which is impossible by
Proposition~\ref{thm:c_not_div}. This proves that no
solution~${ (a, b, c) }$ to equation~\eqref{eq:main}
with~${ \gcd(a, b) = 1 }$ and~${ l = 5 }$ can exist.

\begin{remark}
  It is necessary to first look at the factorization in~${ \Q(v) }$,
  since some of the hyperelliptic curves that come from units we have
  not considered over~${ K }$ have Jacobians with a rank bound that is
  not zero. Furthermore these curves also don't have a local
  obstruction.
\end{remark}



\section{The Frey curves}
\label{sec:FreyCurves}
In this section we construct Frey curves for our problem. A Frey curve
is an elliptic curve that depends on the solution~${ (a, b, c) }$,
which has a mod~${ l }$ Galois representation that is finite flat
at~${ l }$ and unramified at all other primes outside some fixed set
of primes~${ S }$. The set~${ S }$ should be independent of the chosen
solution~${ (a, b, c) }$. Such curves can be realized by a curve which
has a minimal discriminant that is an~${ l }$-th power outside~${ S }$
and which has additive reduction only at primes in~${ S }$.

For our cases we construct such curves using the following fact. Given
two elements~${ B_1 }$ and~${ B_2 }$ of a field~${ k }$ of which their
sum is a square, i.e.~${ B_1 + B_2 = A^2 }$, we can look at the
elliptic curve

\begin{equation*}
  E : y^2 = x^3 + 2 \, A x^2 + B_1 x
\end{equation*}
defined over~${ k }$, for which this model has
discriminant~${ \Delta = 64 \, B_1^2 B_2 }$. Furthermore this curve
can only have additive reduction at primes above~${ 2 }$ and primes
that divide both~${ B_1 }$ and~${ B_2 }$. This is easily verified by
looking at the invariant~${ c_4 = 16 \left( B_1 + 4 \, B_2 \right) }$,
which is coprime to~${ \Delta }$ outside such primes.

Note that this recipe will give us a Frey curve if~${ B_1 }$
and~${ B_2 }$ are coprime~${ l }$-th powers outside the fixed
set~${ S }$. In fact this is the same Frey curve considered for the
generalized Fermat equation with signature~${ (l, l, 2) }$.

Now over the field~${ K = \Q( \sqrt{30} ) }$ we know that we have two
factors~${ g_1(a, b) }$ and~${ g_2(a, b) }$ which are
coprime~${ l }$-th powers outside the set of primes above~${ 2
}$,~${ 3 }$ and~${ 5 }$. Furthermore we have that

\begin{IEEEeqnarray*}{rCrCl}
  \left( \frac{1}{2} - \frac{1}{10} \sqrt{30} \right) g_1(a, b) & + &
  \left( \frac{1}{2} + \frac{1}{10} \sqrt{30} \right) g_2(a, b) & = &
  a^2 \\
  \frac{1}{20} \sqrt{30} \, g_1(a, b) & - &
  \frac{1}{20} \sqrt{30} \, g_2(a, b) & = & b^2,
\end{IEEEeqnarray*}
hence we can apply the construction given above substituting
for~${ B_1 }$ and~${ B_2 }$ the right multiples of~${ g_1(a, b) }$
and~${ g_2(a, b) }$. We will construct the Frey curves we will use
from the four resulting Frey curves
\begin{IEEEeqnarray*}{rCrCrCrCrl}
  E_1' & : & y^2 & = & x^3 & + & 2 \, a x^2 & + & \left( \frac{1}{2} - \frac{1}{10} \sqrt{30} \right) g_1(a, b) x & , \\
  E_1'' & : & y^2 & = & x^3 & + & 2 \, a x^2 & + & \left( \frac{1}{2} + \frac{1}{10} \sqrt{30} \right) g_2(a, b) x & , \\
  E_2' & : & y^2 & = & x^3 & + & 2 \, b x^2 & + & \frac{1}{20} \sqrt{30} \, g_1(a, b) x & \text{, and} \\
  E_2'' & : & y^2 & = & x^3 & + & 2 \, b x^2 & - & \frac{1}{20} \sqrt{30} \, g_2(a, b) x &.
\end{IEEEeqnarray*}
Note that~${ E_1''}$ and~${ E_2'' }$ are Galois conjugates
of~${ E_1 }$ and~${ E_2 }$ respectively, so it suffices to consider
only one of each pair. We pick~${ E_1'' }$ and~${ E_2' }$ and twist
these curves by~${ 30 }$ and~${ 20 }$ respectively to obtain two Frey
curves with an integral model
\begin{IEEEeqnarray*}{rCrCrCrCrl}
  E_1 & : & y^2 & = & x^3 & + & 60 \, a x^2 & + & 30 \left( \left( 15 + 3 \sqrt{30} \right) a^2 + \sqrt{30} \,  b^2 \right) x & \text{, and} \\
  E_2 & : & y^2 & = & x^3 & + & 40 \, b x^2 & + & 20 \left( \sqrt{30} \, a^2 + \left( 10 + 2 \sqrt{30} \right) b^2 \right) x &.
\end{IEEEeqnarray*}
These models have respective discriminants
\begin{IEEEeqnarray*}{rCl}
  \Delta_1 & = & - 2^{9} \cdot 3^{6} \cdot 5^{4} \left( 5 + \sqrt{30} \right) \cdot g_1(a, b) \cdot  g_2(a, b)^2 \text{ and} \\
  \Delta_2 & = & - 2^{13} \cdot 3 \cdot 5^{4} \, \sqrt{30} \cdot g_1(a,b)^2 \cdot g_2(a, b),
\end{IEEEeqnarray*}
${ c_4 }$-invariants
\begin{IEEEeqnarray*}{rCl}
  c_{4,1} & = & - 2^{5} \cdot 3^{2} \cdot 5 \cdot \left( 5 + \sqrt{30} \right) \cdot \left( \left(43 - 8 \, \sqrt{30} \right) a^2 + \left(6 - \sqrt{30}\right) b^2 \right) \text{ and} \\
  c_{4,2} & = & - 2^{6} \cdot 3^{-1} \cdot 5 \cdot \sqrt{30} \cdot \left( 9 a^{2} + \left(18 - 5 \sqrt{30} \right) b^{2} \right),
\end{IEEEeqnarray*}
and~${ j }$-invariants
\begin{IEEEeqnarray*}{rCl}
  j_1(a, b) & = & \left( 11 + 2 \, \sqrt{30} \right) \cdot
  2^6 \cdot \frac{ \left( \left(43 - 8 \, \sqrt{30} \right) a^2 + \left(6 - \sqrt{30}\right) b^2 \right)^3 }{ g_1(a,b) \cdot g_2(a, b)^2 } \text{ and}\\
  j_2(a, b) & = & 2^{6} \cdot 3^{-3} \cdot \frac{ \left( 9 a^{2} + \left(18 - 5 \sqrt{30} \right) b^{2} \right)^{3} }{ g_1(a,b)^2 \cdot g_2(a,b) } .
\end{IEEEeqnarray*}

The~${ j }$-invariants of these elliptic curves are not integral. We
will prove this here as we will need this later on. In particular this
implies that these curves do not have complex multiplication.

\begin{lemma}
  \label{thm:j_integral}
  The ${ j }$-invariants~${ j_1(a, b) }$ and~${ j_2(a, b) }$ are not
  integral. Furthermore there exists a prime of
  characteristic~${ > 5 }$ such that~${ j_1( a, b) }$
  and~${ j_2(a, b) }$ are not integral at that prime.
\end{lemma}
\begin{proof}
  Note that since~${ \gcd(a, b) = 1 }$ the left-hand side of
  equation~\eqref{eq:main} is the sum of at least two non-zero fourth
  powers, hence~${ c > 1 }$. By Proposition~\ref{thm:c_not_div} there
  must be a prime number~${ p > 5 }$ dividing~${ c }$. This implies
  that either~${ g_1(a, b) }$ or~${ g_2(a, b) }$ is divisible by a
  prime above~${ p }$. It thus suffices to prove that the numerators
  of~${ j_1(a, b) }$ and~${ j_2(a, b) }$ are not divisible by the same
  prime.

  Note that the factors
  \begin{equation*}
    \left(43 - 8 \, \sqrt{30} \right) a^2 + \left(6 - \sqrt{30}\right) b^2,
  \end{equation*}
  and
  \begin{equation*}
    9 \, a^{2} + \left( 18 - 5 \, \sqrt{30} \right) b^{2},
  \end{equation*}
  in the numerators of~${ j_1(a, b) }$ and~${ j_2(a, b) }$ are coprime
  with~${ g_1(a, b) }$ and~${ g_2(a, b) }$ outside primes of
  characteristic~${ 2 }$,~${ 3 }$ and~${ 5 }$. This can be easily seen
  by computing the resultants of those polynomials with~${ g_1 }$
  and~${ g_2 }$.
\end{proof}
\begin{corollary}
  \label{thm:non_CM}
  The curves~${ E_1 }$ and~${ E_2 }$ do not have complex
  multiplication.
\end{corollary}
\begin{proof}
  This follows directly from \cite[II, Theorem~6.1]{Silverman2} as
  the~${ j }$-invariants are not integral.
\end{proof}


\section{A Hilbert modular approach}
\label{sec:hilbert}
A natural way of using the Frey curves would be to use the modularity
of elliptic curves over real quadratic fields to prove that there are
Hilbert modular forms which have the same mod~${ l }$ Galois
representation as~${ E_1 }$ or~${ E_2 }$. The level of these newforms
will only depend on certain congruence classes of the chosen solution
and hence all possible candidates can be explicitly computed. It turns
out that the dimension of the corresponding spaces is too high to
perform these computations in a reasonable time. Nevertheless we here
describe the start of this approach.

We first need to compute the conductor of~${ E_1 }$ and~${ E_2 }$ as
the level of the Hilbert modular forms associated to them depends
on it.
\begin{proposition}
  \label{thm:conductorK}
  The conductor of~${ E_1 }$ is
  \begin{equation*}{}
    \curl{N}_1 =
    \begin{cases}
      \ideal{p}_2^{12} \, \ideal{p}_3^{2} \, \ideal{p}_5^{2} \rad_{30} \left( g_1(a, b) g_2(a, b)^2 \right) &
      \text{if } 2 \mid b \\
      \ideal{p}_2^{10} \, \ideal{p}_3^{2} \, \ideal{p}_5^{2} \rad_{30} \left( g_1(a, b) g_2(a, b)^2 \right) &
      \text{if } 2 \nmid b,
    \end{cases}
  \end{equation*}
  and the conductor of~${ E_2 }$ is
  \begin{equation*}
    \curl{N}_2 = \ideal{p}_2^{14} \ideal{p}_5^{2} \rad_{30} \left( g_1(a, b)^2 g_2(a, b) \right),
  \end{equation*}
  where~${ \ideal{p}_2 }$,~${ \ideal{p}_3 }$ and~${ \ideal{p}_5 }$ are
  the unique primes above 2, 3 and 5 respectively and
  where~${ \rad_{30}( N ) }$ is the product of all primes that
  divide~${ N }$ and do not divide~${ 30 }$.
\end{proposition}
\begin{proof}
  This is a computation performed by the code~\cite{code} of the
  author, which makes use of Tate's algorithm. This algorithm is
  implemented for Frey curves by distinguishing for each step in
  Tate's algorithm the congruence classes of the parameters for which
  the algorithm would continue and those for which it would
  stop. These congruence classes can be calculated using Hensel
  lifting as the termination conditions in Tate's algorithm are based
  on certain polynomials in the parameters being zero modulo a
  sufficiently high power of the corresponding prime.
\end{proof}

It has been proven by Freitas, Le Hung and
Siksek~\cite{RealQuadraticModularity} that elliptic curves over real
quadratic fields are modular. In particular the curves~${ E_1 }$
and~${ E_2 }$ are modular. According to~\cite{RealQuadraticModularity}
this means there are Hilbert cuspidal eigenforms~${ f_1 }$
and~${ f_2 }$ over~${ K = \Q( \sqrt{30} ) }$ of parallel weight 2
with rational Hecke eigenvalues such that for all prime
numbers~${ p }$ we have
\begin{equation*}
  \rho_{ E_i, p } \cong \rho_{f_i, p} : G_K \to \GL_2( \Q_p ), \quad i \in \{ 1, 2 \}.
\end{equation*}
Here~${ \rho_{ E_i, p } }$ is the~${ p }$-adic Galois representation
of~${ E_i }$ induced by the Galois action on the Tate
module~${ T_p( E ) }$ and~${ \rho_{ f_i, p } }$ is the~${ p }$-adic
Galois representation associated to~${ f_i }$ by Carayol, Blasius,
Rogawski, Wiles and Taylor.

Note that the conductor of~${ \rho_{ E_i, p } }$ is precisely the
conductor of~${ E_i }$ and the conductor
of~${ \rho_{ f_i, p } }$ is precisely the level of~${ f_i
}$. Therefore we know from Proposition~\ref{thm:conductorK}
that~${ f_1 }$ and~${ f_2 }$ have respective levels~${ \curl{N}_1 }$
and~${ \curl{N}_2 }$.

The levels~${ \curl{N}_1 }$ and~${ \curl{N}_2 }$ are not explicit as
they depend on the chosen solution~${ (a, b, c) }$. However if we
take~${ p = l }$ and look at the mod~${ l }$ Galois
representations~${ \overline{\rho}_{E_i, l} : G_K \to \End( E[l] )
  \cong \GL_2( \F_l ) }$, rather than the~${ p }$-adic representation,
we find that these representations are irreducible. Furthermore they
are finite at all primes not dividing~${ 30 }$. This allows us to
lower the level to a level only divisible by those primes
dividing~${ 30 }$. We prove that the representation is finite here,
but irreducibility needs more results later on and will be proven in
Theorem~\ref{thm:irreducible}.

\begin{proposition}
  \label{thm:unramifiedOutside30}
  The mod~${ l }$ Galois representations
  \begin{equation*}
    \overline{\rho}_{E_i, l} : G_K \to \End( E[l] ) \cong \GL_2( \F_l )
  \end{equation*}
  are finite outside all primes dividing~${ 30 }$. In particular
  they are unramified outside all primes dividing~${ 30 \, l }$.
\end{proposition}
\begin{proof}
  Note that for each finite prime~${ \ideal{p} }$ of~${ K }$ that does
  not divide~${ 30 }$ the order of~${ \ideal{p} }$ in~${ \curl{N}_1 }$
  or~${ \curl{N}_2 }$ is at most one. In case it is zero the curve has
  good reduction at~${ \ideal{ p } }$, hence the mod~${ l }$ Galois
  representation is finite at~${ \ideal{ p } }$. We are left with the
  case the order is one, in which case~${ \ideal{p} }$ must
  divide~${ g_1(a, b) }$ or~${ g_2(a, b) }$. Since~${ g_1(a, b) }$
  and~${ g_2(a, b) }$ are~${ l }$-th powers outside primes
  dividing~${ 30 }$, this implies that the order of~${ \ideal{ p } }$
  in the corresponding discriminant~${ \Delta_i }$ is a multiple
  of~${ l }$. Since the corresponding curve~${ E_i }$ has
  multiplicative reduction at~${ \ideal{ p } }$ the mod~${ l }$ Galois
  representation is finite at~${ \ideal{ p } }$. This is a standard
  result that can be easily proved using the Tate curve
  if~${ \ideal{p} \nmid l }$.
\end{proof}

Using the level lowering result found in~\cite[Theorem 7]{realLL}
for~${ l > 5 }$ we now find that there must be Hilbert cuspidal
eigenforms~${ f_1' }$ and~${ f_2' }$ over~${ K }$ of
parallel weight~2 such
that~${ \overline{\rho}_{ E_i, l } \cong \overline{\rho}_{ f_i',
    \lambda } }$ for~${ i \in \{ 1, 2 \} }$,
where~${ \lambda \mid l }$ is a prime in the coefficient field
of~${ f_i' }$. Here the levels of these newforms are respectively
\begin{IEEEeqnarray*}{rCl}
  \widetilde{\curl{N}}_1 & = &
  \begin{cases}
    \ideal{p}_2^{12} \, \ideal{p}_3^{2} \, \ideal{p}_5^{2}
    &
    \text{if } 2 \mid b \\
    \ideal{p}_2^{10} \, \ideal{p}_3^{2} \, \ideal{p}_5^{2} &
    \text{if } 2 \nmid b
  \end{cases} \\
  \widetilde{\curl{N}}_2 & = & \ideal{p}_2^{14} \ideal{p}_5^{2}.
\end{IEEEeqnarray*}

The strategy would now be to compute all cuspidal Hecke eigenforms
over~${ K }$ of parallel weight~2 and
levels~${ \widetilde{\curl{N}}_1 }$ and~${ \widetilde{\curl{N}}_2 }$
and prove that none of them can have a mod~${ \lambda \mid l }$
representation isomorphic to the mod~${ l }$ representation of the
corresponding~${ E_i }$. This would prove a contradiction, hence the
implicit assumption that a primitive solution~${ (a, b, c) }$ to
equation~\eqref{eq:main} would be false as we want.

\begin{remark}
  Besides working with the curves~${ E_1 }$ and~${ E_2 }$ one might
  also want to work with curves that are isomorphic
  over~${ \overline{\Q } }$. In case we do not want to change the
  field~${ K }$ over which the curves are defined, these isomorphic
  curves would be twists of our original curves. Note that twisting by
  an element~${ \gamma \in K^* }$ can only change the conductor of the
  curve at a prime~${ \ideal{ p } }$ if the corresponding field
  extension~${ K( \sqrt{ \gamma } ) }$ is ramified
  at~${ \ideal{ p } }$. The values~${ \gamma \in K^* }$ for
  which~${ K( \sqrt{ \gamma } ) }$ is unramified at a fixed
  prime~${ \ideal{ p } }$ form a
  subgroup~${ H_{\ideal{p}} \subseteq K^* }$ of which the
  quotient~${ K^* / H_{\ideal{p}} }$ is finite. Limiting ourselves to
  the primes dividing~${ \widetilde{\curl{N}}_1 }$
  and~${ \widetilde{\curl{N}}_2 }$ we thus only have a finite
  computation to see if these levels can be made any smaller.

  It turns out that by twisting we can get the lowest
  level
  \begin{equation*}
    \widetilde{\curl{N}}_1 =
    \begin{cases}
      \ideal{p}_2^{12} \, \ideal{p}_3^{2} \, \ideal{p}_5^{2} &
      \text{if } 2 \mid b \\
      \ideal{p}_2^{4} \, \ideal{p}_3^{2} \, \ideal{p}_5^{2} &
      \text{if } 2 \nmid b \text{ and } a \equiv 1 \text{ (mod } 4 \text{)} \\
      \ideal{p}_3^{2} \, \ideal{p}_5^{2} &
      \text{if } 2 \nmid b \text{ and } a \equiv 3 \text{ (mod } 4 \text{)},
    \end{cases}
  \end{equation*}
  in case we twist~${ E_1 }$ with~${ 6 + \sqrt{30} }$. If we twist the
  curve~${ E_1 }$ by~${ -6 - \sqrt{30} }$ we get the same level, but
  with the latter two conditions interchanged.
\end{remark}

Using Magma we quickly find that the dimension of some of the sought
spaces of newforms would be way too large to compute in. For example
using the levels of the untwisted curves the dimension of the smallest
space is~${ 206 \, 720 }$, which is way beyond the largest
computational examples done in the literature. We can only do better
in case~${ 2 \nmid b }$ where the twisted curve in the remark gives us
a space of dimension~${ 542 }$ for the newforms of
level~${ \ideal{p}_3^2 \, \ideal{p}_5^2 }$. A lower level for the
case~${ 2 \mid b }$ is lacking though, making this insufficient to prove
Theorem~\ref{thm:main} completely.


\section{${ \Q }$-curves}
\label{sec:Qcurves}
In this section we use the modularity of~${ \Q }$-curves to prove the
non-existence of solutions. This technique has been applied to other
Diophantine equations in works such as \cite{DieulefaitFreitas},
\cite{DieulefaitUrroz}, \cite{BennettChen}, \cite{Chen}, \cite{Chen1},
\cite{BennettChenDahmenYazdani}, and \cite{Ellenberg}. The approach
here is similar to the one in the mentioned articles, leaning heavily
on the work by Quer \cite{Quer}. It differs in some crucial points,
where we will give an algorithmic approach that works in a general
context.

Look back at our original curve
\begin{equation*}
  E : y^2 = x^3 + 2 A x^2 + B_1 x,
\end{equation*}
where~${ A^2 = B_1 + B_2 }$. Note that by construction this curve has
a 2-torsion point and hence an obvious~${ 2 }$-isogeny defined
over~${ \Q }$. From~\cite[III, example 4.5]{Silverman1} we deduce that
the image of this~${ 2 }$-isogeny is
\begin{equation*}
  \tilde{E} : y^2 = x^3 - 4 A x^2 + \left( 4 A^2 - 4 B_1 \right) x = x^3 - 4 A x^2 + 4 B_2 x,
\end{equation*}
which is a twist by~${ -2 }$ of the complementary curve
\begin{equation*}
  E' : y^2 = x^3 + 2 A x^2 + B_2 x.
\end{equation*}
In particular such a curve is thus~${ 2 }$-isogenous over a field
extension containing~${ \sqrt{-2} }$ to the curve in which the roles
of~${ B_1 }$ and~${ B_2 }$ are swapped.

Note that for the Frey curves~${ E_1 }$ and~${ E_2 }$ we constructed,
the chosen~${ B_1 }$ and~${ B_2 }$ were Galois conjugates of one
another, whilst~${ A }$ was rational. This implies that the
curves~${ E_1 }$ and~${ E_2 }$ are 2-isogenous to their Galois
conjugates over~${ K( \sqrt{ -2 } ) }$. This means that even though
the curves are not defined over~${ \Q }$, their isogeny class
over~${ \overline{\Q} }$ is. Curves over~${ \overline{\Q} }$ for which
their isogeny class is defined over~${ \Q }$ are
called~\emph{${ \Q }$-curves}.

Ribet \cite{Ribet} proved, using the Serre conjectures,
that~${ \Q }$-curves, which do not have complex multiplication, are
modular in the sense that they are isogenous to a quotient
of~${ J_1(N) }$ for some integer level~${ N > 0 }$. This form of
modularity gives us classical modular forms associated to our
elliptic curves.

The proof of Ribet makes use of abelian varieties of~${ \GL_2
}$-type. An abelian variety~${ A }$ over~${ \Q }$ is called of
${ \GL_2 }$-type if~${ \End{ A } \otimes \Q }$ contains a
number field of degree equal to~${ \dim A }$. Such a variety is
also~${ \Q }$-simple if and only if~${ \End{ A } \otimes \Q }$ is such
a number field (\cite[Theorem2.1]{Ribet}).

${ \Q }$-simple abelian varieties of~${ \GL_2 }$-type naturally arise
as the quotients of~${ J_1( N ) }$ associated to
Hecke eigenforms of weight 2. One dimensional quotients
over~${ \overline{\Q} }$ of such varieties are
naturally~${ \Q }$-curves. Ribet proved that~${ \Q }$-curves without
complex multiplication are always isogenous to such a quotient
(\cite[Theorem 6.1]{Ribet}). Furthermore he shows that~${ \Q }$-simple
abelian varieties of~${ \GL_2 }$-type are isogenous to a quotient
of~${ J_1(N) }$ for some~${ N > 0 }$ if the Serre modularity
conjectures hold (\cite[Theorem 4.4]{Ribet}).

We want to apply this theory to our~${ \Q }$-curves~${ E_1 }$
and~${ E_2 }$. In particular we want to explicitly compute
levels~${ N_1 }$ and~${ N_2 }$ for which~${ E_1 }$ and~${ E_2 }$ are
isogenous to quotients of~${ J_1(N_1) }$ and~${ J_1(N_2) }$
respectively, such that we can compute the associated modular
forms. For this we will need to work more explicitly with the abelian
varieties of~${ \GL_2 }$-type of which our curves will be
quotients. We will use results of Quer \cite{Quer} to compute these.

\begin{remark}
  In the papers of Ribet and Quer \cite{Ribet,Quer}
  all~${ \Q }$-curves considered are without complex
  multiplication. For ease of writing we shall also assume this for
  all~${ \Q }$-curves mentioned. This is fine since by
  Corollary~\ref{thm:non_CM} the curves~${ E_1 }$ and~${ E_2 }$ do not
  have complex multiplication.
\end{remark}

\subsection{Basic invariants}

We will use the following main result by Quer.
\begin{proposition} \label{thm:Quer5.1} \cite[Proposition~5.1]{Quer}
  Let~${ E }$ be a~${ \Q }$-curve for which both~${ E }$ and the
  isogenies defining the~${ \Q }$-curve structure are defined over
  some Galois number field~${ F }$. Let~${ B = \resultant_{\Q}^{F} E }$
  be the restriction of scalars. If~${ \End_\Q( B ) }$ is a
  commutative algebra, then~${ B }$ factors over~${ \Q }$ as a product
  of~${ \Q }$-simple mutually non~${ \Q }$-isogenous abelian varieties
  of~${ \GL_2 }$-type.
\end{proposition}
In order to apply this result to our curves~${ E_1 }$ and~${ E_2 }$ we
will first recap some general theory.

Let~${ E }$ be~${ \Q }$-curve defined by
isogenies~${ \phi_\sigma : \prescript{\sigma}{}{E} \to E }$. To such a
curve we can associate the \emph{degree
  map}
\begin{equation*}
  d : G_\Q \to \Q^*, \quad \sigma \mapsto \deg \phi_\sigma,
\end{equation*}
and a 2-cocycle~${ c_E : G_\Q^2 \to \Q^* }$ given by
\begin{equation}
  \label{eq:cdef}
  c_E( \sigma, \tau ) = \phi_{\sigma} \prescript{\sigma}{}{\phi_{\tau}} \phi_{\sigma \tau}^{-1},
\end{equation}
where~${ \phi_{\sigma \tau}^{-1} }$ is the dual
of~${ \phi_{\sigma \tau} }$ divided by its degree. Since~${ E }$ has
no complex multiplication the element on the right can be considered
as a non-zero element of~${ \End( E ) \otimes \Q \cong \Q }$, hence
the equality makes sense. By taking degrees in
equation~\eqref{eq:cdef} it is obvious that~${ c_E^2 }$ is the
coboundary of~${ d }$. Furthermore Proposition~2.1 of \cite{Quer}
tells us that~${ d : G_\Q \to \Q^* / \left( \Q^* \right)^2 }$
and~${ \xi(E) = [ c_E ] \in H^2 ( G_\Q, \Q^* ) }$ are invariants of
the isogeny class of~${ E }$.

Note that~${ d : G_\Q \to \Q^* / \left( \Q^* \right)^2 }$ has trivial
coboundary, hence is a homomorphism. Its fixed field~${ K_d }$ is by
Remark~2.6.b of \cite{Quer} the smallest field over which a curve
isogenous to~${ E }$ can be defined.

We say that~${ E }$ is \emph{completely defined} over some
field~${ F }$ if it is defined over~${ F }$ and all~${ \phi_\sigma }$
are also defined over~${ F }$. According to Remark~2.6.a of
\cite{Quer} there are curves isogenous to~${ E }$ that are completely
defined over the field~${ K_d( \sqrt{\pm d_1}, \ldots, \sqrt{\pm d_n} ) }$
for~${ \{ d_1, \ldots, d_n \} }$ a minimal set of generators of the
image of~${ d }$ in~${ \Q^* / \left( \Q^* \right)^2 }$.

A \emph{splitting map} is a continuous
map~${ \beta : G_\Q \to \overline{\Q}^* }$ of which the coboundary
with respect to the trivial action on~${ \overline{ \Q } }$ is
precisely~${ c_E }$. If we consider such a~${ \beta }$ as a map
to~${ \overline{\Q}^* / \Q^* }$ it has trivial coboundary, hence it is
a homomorphism. We will call the fixed field~${ K_\beta }$ a
\emph{splitting field} of~${ E }$.

For a splitting map~${ \beta }$ the coboundary of the
map~${ \varepsilon = \beta^2 / d }$ is~${ c_E^2 / c_E^2 = 1 }$,
hence~${ \varepsilon }$ is a character called a \emph{splitting
  character} of~${ E }$. Its fixed field~${ K_\varepsilon }$ plays an
important role as~${ K_\beta = K_\varepsilon K_d }$.

A \emph{dual basis} for the degree map~${ d }$ is a
set~${ \{ (a_1, d_1), \ldots, (a_n, d_n) \} \subseteq \Z^2 }$ such that
\begin{itemize}
\item ${ K_d = \Q( \sqrt{a_1}, \ldots \sqrt{a_n} ) }$
  and~${ [K_d : \Q] = 2^n }$, and
\item Each~${ d_i }$ is the image of~${ d }$ at
  some~${ \sigma_i \in G_{\Q}^{K_d} }$, where~${ \sigma_i }$
  satisfies
  \begin{equation*}
    \prescript{\sigma_i}{}{\sqrt{a_j}} = (-1)^{\delta_{i j}} \sqrt{a_j}
  \end{equation*}
\end{itemize}
Theorem~3.1 of \cite{Quer} tells us that the sign
component~${ \xi_{\pm}(E) }$ of~${ \xi(E) }$, i.e. the part generated
by the sign of~${ c_E }$, can be described as
\begin{equation*}
  \xi_{\pm}(E) = \prod_{i=1}^n (a_i, d_i),
\end{equation*}
where~${ (a, d) }$ is the quaternion algebra over the rationals defined
by~${ i^2 = a}$,~${ j^2 = d }$ and~${ ij = k = -ji }$
inside~${ \Brauer_2( \Q ) = H^2( G_\Q, \{ \pm 1 \} ) }$.

This description of~${ \xi_{\pm}( E ) }$ in terms of a dual basis
gives us a condition to see if a curve isogenous to~${ E }$ can be
completely defined over~${ K_d }$
\cite[Corollary~3.3]{Quer}. Furthermore combined with the discussion
on page~302 of \cite{Quer} we know that a
character~${ \varepsilon : G_\Q \to \overline{\Q}^* }$ is a splitting
character if and only if
\begin{equation*}
  \varepsilon_p(-1) = \prod_{i=1}^n (a_i, d_i)_p,
\end{equation*}
for every prime number~${ p }$. Here~${ \varepsilon_p }$ is
the~${ p }$-part of~${ \varepsilon }$ considered as a Dirichlet
character and~${ (a_i, d_i)_p }$ is a Hilbert symbol.

A \emph{decomposition field} for~${ E }$ is an abelian number field
that contains both a field over which~${ E }$ is completely defined
and a splitting field of~${ E }$. Proposition~5.2 of~\cite{Quer} tells
us that the condition that~${ \End_\Q( B ) }$ is commutative in
Proposition~\ref{thm:Quer5.1} can be replaced by the condition
that~${ F }$ is a decomposition field and the existence of some
splitting map defined over~${ G_{\Q}^F }$. Furthermore
\cite[Proposition~5.2]{Quer} states that if~${ F }$ is a decomposition
field, the latter is always true for some curve isogenous to~${ E }$
over~${ \overline{ \Q } }$.

We now explicitly compute all these quantities for
our~${ \Q }$-curves~${ E_1 }$ and~${ E_2 }$.
\begin{proposition}
  \label{thm:fieldsEi}
  For both~${ E_1 }$ and~${ E_2 }$ we have the same data listed below.
  \begin{itemize}
  \item The degree map~${ d : G_\Q \to \Q^* }$ given by
    \begin{equation*}
      d( \sigma ) =
      \begin{cases}
        1 & \text{if } \sigma \in G_K \\
        2 & \text{if } \sigma \not\in G_K.
      \end{cases}
    \end{equation*}
  \item The 2-cocycle~${ c : G_\Q^2 \to \Q^* }$ given by
    \begin{equation*}
      c( \sigma, \tau ) =
      \begin{cases}
        1 & \text{if } \sigma \in G_{K( \sqrt{-2} ) } \text{ or } \tau \in G_K \\
        -1 & \text{if } \tau \not\in G_K \text{, } \prescript{\sigma}{}{\sqrt{-2}}
        = -\sqrt{-2} \text{ and } \prescript{\sigma}{}{\sqrt{30}} = \sqrt{30}, \\
        -2 & \text{if } \tau \not\in G_K \text{, } \prescript{\sigma}{}{\sqrt{-2}}
        = \sqrt{-2} \text{ and } \prescript{\sigma}{}{\sqrt{30}} = -\sqrt{30}, \\
        2 & \text{if } \tau \not\in G_K \text{, } \prescript{\sigma}{}{\sqrt{-2}}
        = -\sqrt{-2} \text{ and } \prescript{\sigma}{}{\sqrt{30}} = -\sqrt{30},
      \end{cases}
    \end{equation*}
  \item The field~${ K_d = K = \Q( \sqrt{30} )}$ over which the curves are defined.
  \item The field~${ K( \sqrt{-2} ) = K_d( \sqrt{-2} ) }$ over which
    the curves are completely defined.
  \item A dual basis~${ \{ ( 30, 2 ) \} }$.
  \item A splitting
    character~${ \varepsilon : G_\Q \to \overline{\Q}^* }$, that as a
    Dirichlet character is one of the characters of conductor 15 and
    order 4, with corresponding fixed
    field~${ K_\varepsilon = \Q( \zeta_{15} + \zeta_{15}^{-1} ) }$ of
    degree~4.
  \item A splitting
    field~${ K_\beta = K( \zeta_{15} + \zeta_{15}^{-1} ) = \Q(
      \sqrt{6}, \zeta_{15} + \zeta_{15}^{-1} ) }$ of degree 8.
  \item A decomposition
    field~${ K_{\text{dec}} = \Q( \sqrt{-2}, \sqrt{-3}, \zeta_{15} +
      \zeta_{15}^{-1} ) }$ of degree 16.
  \end{itemize}
\end{proposition}
\begin{proof}
  All this data can be computed from the
  isogenies~${ \phi_{\sigma} : \prescript{\sigma}{}{E_i} \to E_i }$,
  which we can take to be the identity if~${ \sigma \in G_K }$ and
  the~${ 2 }$-isogeny over~${ K( \sqrt{-2} ) }$ described before
  otherwise. Note that the latter can be explicitly described using
  the formula in~\cite[III, example 4.5]{Silverman1} and the map
  scaling with~${ \sqrt{-2} }$. In fact we compute from this by hand
  the degrees of these isogenies and the one-cocycle discussed on
  pages 288 and 289 of \cite{Quer} from which the code~\cite{code} can
  compute all the other data.
\end{proof}

\begin{remark}
  Note that the field~${ K( \sqrt{-2} ) }$ is a minimal field over
  which a curve isogenous to~${ E_1 }$ or~${ E_2 }$ can be completely
  defined. This is easily verified through Corollary~3.3 in
  \cite{Quer}, as it excludes the case that an isogenous curve can be
  defined over the field~${ K_d = K }$ which is the only possible
  smaller field. For this one checks that
  \begin{equation*}
    ( 30, 2 ) \ne 1 \quad \text{and} \quad ( 30, 2 ) \ne ( -1, 30 ),
  \end{equation*}
  inside~${ \Brauer_2( \Q ) }$, which can be verified by checking the
  corresponding Hilbert symbols at~${ 5 }$.
\end{remark}


\subsection{A decomposable twist}

As mentioned Proposition~5.2 in \cite{Quer} tells us that if~${ F }$
is a decomposition field, then some curve isogenous to~${ E_i }$
satisfies all the conditions of Proposition~\ref{thm:Quer5.1}. We thus
know that there are curves isogenous to~${ E_1 }$ and~${ E_2 }$ of
which the restrictions of scalars over~${ K_{\text{dec}} }$ decompose
as the product of~${ \Q }$-simple abelian varieties
of~${ \GL_2 }$-type.

Our goal now is to find these curves isogenous to~${ E_1 }$
and~${ E_2 }$ for which the result of Proposition~\ref{thm:Quer5.1}
applies. According to Proposition~5.2 of~\cite{Quer} these should be
curves for which the class of the 2-cocycle
in~${ H^2( G_{\Q}^{K_{\text{dec}}}, \Q^* ) }$ is the same as the
coboundary of the chosen splitting map. It is already discussed on
page~297 of \cite{Quer} that these curves can be obtained as twists of
the original curves by some element~${ \gamma \in K_{\text{dec}}
}$. This result depends on a certain embedding problem being
unobstructed as proven in \cite{Quer2}.

In most articles in which Frey~${ \Q }$-curves are used a correct
twist is found by making an educated guess or a small search using
results from \cite{Quer2}, for example in \cite{DieulefaitFreitas},
\cite{DieulefaitUrroz}, \cite{Chen1}, and \cite{Chen}. We present here
an algorithmic approach that always works based on the approach in
\cite{BennettChen}.

We first explain the approach for a general~${ \Q }$-curve~${ E
}$. Let notation be as before and let~${ F }$ be a decomposition field
of~${ E }$ with corresponding splitting map~${ \beta }$.
If~${ c_\beta }$ is the coboundary of~${ \beta }$ with respect to
trivial action on~${ \overline{\Q}^* }$, we can look at the
2-cocycle~${ c_\beta / c_E }$. By correctly rescaling~${ \beta }$ if
necessary this can be interpreted as an element
of~${ H^2( G_{\Q}^{F} , \{ \pm 1 \} ) }$, hence it defines an
embedding problem of Galois groups as noted in \cite[page 297]{Quer}.

The solutions~${ F( \sqrt{ \gamma } ) }$ to this embedding problem
give~${ \gamma \in F^* }$ that perform the sought twist. Furthermore
these are the~${ \gamma }$ that satisfy
\begin{equation*}
  \prescript{\sigma}{}{\gamma} = \alpha(\sigma)^2 \gamma \text{ for all } \sigma \in G_{\Q}^{F},
\end{equation*}
for some~${ \alpha : G_{\Q}^{F} \to F^* }$ of which the coboundary
is~${ c_\beta / c_E }$. By finding
an explicit~${ \alpha : G_{\Q}^{F} \to F^* }$ of which the coboundary
is~${ c_\beta / c_E }$ we can find such a~${ \gamma \in F^* }$ by
applying Hilbert 90 to~${ \alpha^2 }$.

The trick used in \cite{BennettChen} is to find such~${ \alpha }$ that
take values in the finitely generated group~${ \curl{O}_F^* }$. This
makes finding~${ \alpha }$ a linear problem in the exponents with the
equations
\begin{equation*}
  \alpha( \sigma ) \prescript{\sigma}{}{\alpha( \tau )} \alpha(\sigma \tau)^{-1}
  = \frac{c_\beta( \sigma, \tau )}{c_E( \sigma, \tau ) } \quad \sigma, \tau \in G_{\Q}^{F}.
\end{equation*}
In general such an~${ \alpha }$ need not exist, but we can make a more
precise statement. Since any solution~${ \gamma }$ in the embedding
problem may be changed by a square, we can limit the primes appearing
in~${ \prescript{\sigma}{}{\gamma} \gamma^{-1} }$ to those generating
the class group and their conjugates. This implies we can find a
finite set~${ S }$ of primes of~${ F }$ for which an~${ \alpha }$ with
values in the~${ S }$-units~${ \curl{O}_{F, S}^* }$ exists if any
solution exists. In particular if the class number is one, such
an~${ \alpha }$ exists with values in~${ \curl{O}_{F}^* }$. In any
case~${ \curl{O}_{F, S}^* }$ is finitely generated so the problem is
solvable using linear algebra. The author has results that make the
set~${ S }$ much more precise which have been incorporated in the
code~\cite{code} to work for arbitrary fields~${ F }$. A computation
in SAGE \cite{sage} shows that the class group of~${ K_{\text{dec}} }$ is
trivial, so we will omit these results for this case.

\begin{remark}
  We can be very explicit about how unique the
  twist~${ \gamma \in F^* }$ is. Combining the long exact
  sequences over~${ G_{\Q}^{F} }$ for the short exact
  sequences
  \begin{equation*}
    \begin{tikzcd}
      1 \arrow[r] &
      \{ \pm 1 \} \arrow[r] &
      F^* \arrow[r, "\cdot^2"] &
      \left( F^* \right)^2 \arrow[r] &
      1,
    \end{tikzcd}
  \end{equation*}
  and
  \begin{equation*}
    \begin{tikzcd}
      1 \arrow[r] &
      \left( F^* \right)^2 \arrow[r] &
      F^* \arrow[r] &
      F^* / \left( F^* \right)^2 \arrow[r] &
      1,
    \end{tikzcd}
  \end{equation*}
  we get a diagram of the form
  \begin{equation*}
    \begin{tikzcd}[column sep=small]
      & &
      H^1 \left( G_{\Q}^{F}, F^* \right) = 1 \arrow[d] \\
      \Q^* \arrow[r] &
      \left( F^* / \left( F^* \right)^2 \right)^{G_{\Q}^{F}} \arrow[r] &
      H^1 \left( G_{\Q}^{F}, \left( F^* \right)^2 \right) \arrow[r] \arrow[d] &
      H^1 \left( G_{\Q}^{F}, F^* \right) = 1 \\ & &
      H^2 \left( G_{\Q}^{F}, \{ \pm 1 \} \right).
    \end{tikzcd}
  \end{equation*}
  Here~${ H^1 \left( G_{\Q}^{F}, F^* \right) = 1
  }$ follows from Hilbert 90. Now the class of the
  sought~${ \gamma }$ lives
  in~${ \left( F^* / \left( F^* \right)^2
    \right)^{G_{\Q}^{F}} }$, that is all
  those~${ \gamma \in F^* }$ for
  which~${ \prescript{\sigma}{}{\gamma} \gamma^{-1} \in \left(
      F^* \right)^2 }$ for
  all~${ \sigma \in G_{\Q}^{F} }$. Furthermore we see
  that two such~${ \gamma }$ define the same resulting class
  in~${ H^2 \left( G_{\Q}^{F}, \{ \pm 1 \} \right) }$
  if and only if they differ by an element
  of~${ \Q^* \left( F^* \right)^2 }$.
\end{remark}

We now apply this theory to~${ E_1 }$ and~${ E_2 }$.
Since~${ K_{\text{dec}} }$ has class number~${ 1 }$ we can search for
an~${ \alpha : G_{\Q}^{ K_{\text{dec}} } \to
  \curl{O}_{K_{\text{dec}}}^* }$. The code~\cite{code} does this to
find suitable twists of~${ E_1 }$ and~${ E_2 }$. As remarked above we
can change the twist by a square, which we do to find one for which
the twist parameter has a smaller minimal polynomial. In fact we find
that for~${ \gamma \in K_{ \text{dec} } }$ any root of the polynomial
\begin{equation*}
  x^{8} - 40 x^{7} - 550 x^{6} - 1840 x^{5} - 285 x^{4} + 3600 x^{3} - 1950 x^{2} + 200 x + 25,
\end{equation*}
the twists~${ E_{1, \gamma} }$ and~${ E_{2, \gamma} }$ of~${ E_1 }$
and~${ E_2 }$ by~${ \gamma }$ satisfy the conditions of
Proposition~\ref{thm:Quer5.1}. Note that this polynomial is not
irreducible over~${ K }$, but still any choice of root will suffice.

Note that by twisting the curves, we must also change the isogenies
that define the~${ \Q }$-curve structure.  On page~291 of \cite{Quer}
it is made explicit how these isogenies change under twists. The
code~\cite{code} computes these new isogenies automatically. As the
isogenies differ and~${ \Q( \gamma ) = K_{\beta} }$, the twisted
curves~${ E_{1, \gamma} }$ and~${ E_{2, \gamma} }$ are completely
defined over~${ K_{ \beta } }$. Since the same splitting maps also
work for the twisted curves, this means~${ K_{ \beta } }$ is a
decomposition field for~${ E_{1, \gamma} }$ and~${ E_{2, \gamma} }$.

Applying Proposition~\ref{thm:Quer5.1} with the facts above we get the
following proposition.
\begin{proposition}
  \label{thm:twistEi}
  For each~${ i \in \{ 1, 2 \} }$ the abelian variety
  ${ \resultant_{\Q}^{K_\beta} E_{i, \gamma} }$ is~${ \Q }$-isogenous to
  a product of~${ \Q }$-simple, mutually non~${ \Q }$-isogenous
  abelian varieties of~${ \GL_2 }$-type.
\end{proposition}

In order to work with the abelian varieties of~${ \GL_2 }$-type that
arise in Proposition~\ref{thm:twistEi} we need to compute some
explicit data about these varieties. In particular we will need the
dimension and endomorphism ring of each factor. For this we need some
more theory which we again first discuss for a
general~${ \Q }$-curve~${ E }$.

Let~${ F }$ be a decomposition field for~${ E }$. As discussed in
Section~5 of \cite{Quer}, the
isogenies~${ \phi_\sigma : \prescript{\sigma}{}{E} \to E }$ induce
endomorphisms~${ \Phi_\sigma }$ on~${ \resultant_{\Q}^{F} E }$ that are
defined over~${ \Q }$. If~${ \resultant_{\Q}^{F} E }$ also decomposes
into abelian varieties~${ A }$ of~${ \GL_2 }$-type,
these~${ \Phi_\sigma }$ induce non-zero
elements~${ \beta( \sigma ) \in \End A \otimes \Q }$ on each~${ A }$.
Since the ring~${ \End A \otimes \Q }$ is a number field, these give
rise to maps~${ \beta : G_\Q^{F} \to \overline{\Q}^* }$. Since
\begin{equation*}
  \phi_{\sigma} \prescript{\sigma}{}{\phi_{\tau}} \phi_{\sigma \tau}^{-1} = c_E( \sigma, \tau )
  \text{ for all } \sigma, \tau \in G_{\Q}^{\F},
\end{equation*}
we find that
\begin{equation*}
  \beta( \sigma ) \beta( \tau ) \beta( \sigma \tau )^{-1} = c_E( \sigma, \tau )
  \text{ for all } \sigma, \tau \in G_{\Q}^{\F},
\end{equation*}
so each~${ \beta }$ is a splitting map.

In this setting we also have
that~${ R = \End_\Q \left( \resultant_{\Q}^{F} E \right) \otimes \Q }$ is generated
by~${ \Phi_\sigma }$ as a~${ \Q }$-vector space. Given a splitting
map~${ \beta : G_{\Q}^{F} \to \overline{\Q}^* }$ we get a
ring homomorphism
\begin{equation*}
  R \to \overline{\Q}, \quad \Phi_\sigma \mapsto \beta( \sigma ),
\end{equation*}
of which the kernel defines a~${ \Q }$-simple factor~${ A }$
of~${ \resultant_{\Q}^{F} E }$. This and the fact above gives a
correspondence between splitting
maps~${ G_{\Q}^{F} \to \overline{\Q}^* }$ and the factors
of~${ \resultant_{\Q}^{F} E }$.

Note that any two splitting
maps~${ \beta_1, \beta_2 : G_{\Q}^{F} \to \overline{\Q}^* }$ have the
same coboundary, hence the
difference~${ \chi = \beta_2 \beta_1^{-1} }$ is a character.  Since
there are only finitely many characters on~${ F }$ we can compute all
such splitting maps by computing a single one, which we already did to
compute a decomposition field~${ F }$.

Two distinct splitting maps might however correspond to
different~${ \Q }$-simple factors of~${ \resultant_{\Q}^{F} E }$ which
are isogenous over~${ \Q }$. According to Lemma~5.3 of \cite{Quer}
this is the case if and only if the two splitting maps are Galois
conjugates of one another. To compute the~${ \Q }$-simple factors
of~${ \GL_2 }$-type it thus suffices to compute all Galois conjugacy
classes of splitting maps~${ G_{\Q}^F \to \overline{\Q}^*}$.

Note that by the construction of an abelian variety~${ A }$ from a
splitting map~${ \beta }$, the ring~${ \End A \otimes \Q }$ must be
the smallest number field containing the values of~${ \beta
}$. According to Proposition~4.1 in \cite{Quer} this field can be
explicitly computed from the corresponding splitting character and a
dual basis of the degree map. Since~${ A }$ is~${ \Q }$-simple and
of~${ \GL_2 }$-type the degree of this field is also the dimension
of~${ A }$.

Performing the necessary computations we can improve
Proposition~\ref{thm:twistEi} as follows.
\begin{theorem}
  \label{thm:EiDecomposition}
  Let~${ i \in \{ 1, 2 \} }$. We have that
  \begin{equation*}
    \resultant_{\Q}^{K_{\beta}} E_{i, \gamma} \text{ is } \Q \text{-isogenous to } A_{i, 1} \times A_{i, 2},
  \end{equation*}
  where
  \begin{itemize}
  \item each ${ A_{i, j} }$ is a~${ \Q }$-simple abelian variety
    of~${ GL_2 }$-type over~${ \Q }$ of dimension 4
    with~${ \End A_{i, j} \otimes \Q \cong L_\beta = \Q(\zeta_8) }$, and
  \item the varieties~${ A_{i, 1} }$ and~${ A_{i, 2} }$ are not
    isogenous over~${ \Q }$.
  \end{itemize}
\end{theorem}


\subsection{Modularity of~${ \Q }$-curves}
\label{sec:QcurveModularity}

Since the Serre conjectures have been proven by Khare and Wintenberger
in \cite{KhareWintenberger1} and \cite{KhareWintenberger2} we can now
use the modularity result proven by Ribet in \cite{Ribet}. In order to
be explicit about the level and character of the corresponding
newforms we need to discuss some more theory.

Let~${ E }$ be a~${ \Q }$-curve as before and
assume~${ B := \residue_{\Q}^{F} E }$ decomposes into abelian
varieties of~${ \GL_2 }$-type. If~${ A }$ is a~${ \Q }$-simple factor
of~${ B }$ with corresponding splitting
map~${ \beta : G_{\Q}^{F} \to \overline{\Q}^* }$, we have a
commutative diagram
\begin{equation*}
  \begin{tikzcd}
    E \arrow[r, "\sigma"] &
    \prescript{\sigma}{}{E} \arrow[r, "\phi_{\sigma}"] &
    E \\
    B \arrow[r, "\sigma"] \arrow[u] \arrow[d] &
    B \arrow[r, "\Phi_\sigma"] \arrow[u] \arrow[d] &
    B \arrow[u] \arrow[d] \\
    A \arrow[r, "\sigma"] &
    A \arrow[r, "\beta(\sigma)"] &
    A,
  \end{tikzcd}
\end{equation*}
for each~${ \sigma \in G_{\Q} }$. Since the
map~${ \Phi_\sigma \circ \sigma }$ on~${ B }$ is completely determined
by~${ \phi_\sigma \circ \sigma }$ on~${ E }$, we find that the
map~${ \beta(\sigma) \circ \sigma }$ on~${ A }$ is also completely
fixed by this. In particular~${ \beta(\sigma) \circ \sigma }$ acts on
the Tate module~${ V_p( A ) }$ as~${ \phi_\sigma \circ \sigma }$ acts
on~${ V_p( E ) }$ for each prime number~${ p }$.

Let~${ L_A := \End A \otimes \Q }$. As argued at the start of
Section~2 in \cite{Ribet} the Tate module~${ V_p( A ) }$ is a free
module over~${ L_A \otimes \Q_p }$ of rank 2 and hence the action
of~${ G_\Q }$ on it decomposes as a product of 2-dimensional
representations~${ \rho_{A, \ideal{p}} : G_\Q \to \GL_2( L_{A,
    \ideal{p}} ) }$ for~${ \ideal{p} \mid p }$ a prime of~${ L_A
}$. By the argument above~${ \sigma }$ acts
on~${ V_p( A ) \cong \left( L_A \otimes \Q_p \right)^2 }$
as~${ \beta(\sigma)^{-1} }$ times the action
of~${ \phi_{\sigma} \circ \sigma }$ on~${ V_p( E ) }$,
hence~${ \rho_{A, \ideal{p}}( \sigma ) \sim \beta(\sigma)^{-1} \,
  \rho_{E, p}(\sigma) }$ if~${ \phi_\sigma = 1 }$. Here~${ \sim }$
means the matrices are conjugate to each other. This is in particular
the case for all~${ \sigma \in G_{F} }$.

\label{sec:galoisDiscussion}

\begin{remark}
  If there is an isogeny~${ \phi : E \to E' }$ defined
  over~${ \overline{ \Q } }$, then by choosing specific
  isogenies~${ \phi_{\sigma}' : \prescript{\sigma}{}{E'} \to E }$ we
  have a commutative diagram
  \begin{equation*}
    \begin{tikzcd}
      E \arrow[r, "\sigma"] \arrow[d, "\phi"] &
      \prescript{\sigma}{}{E} \arrow[r, "\phi_{\sigma}"] \arrow[d, "\prescript{\sigma}{}{\phi}"]&
      E \arrow[d, "\phi"]\\
      E' \arrow[r, "\sigma"] &
      \prescript{\sigma}{}{E'} \arrow[r, "\phi_{\sigma}'"] &
      E',
    \end{tikzcd}
  \end{equation*}
  for all~${ \sigma \in G_{\Q} }$. This implies
  that~${ \beta( \sigma ) \circ \sigma }$ acts on~${ V_p( A ) }$
  as~${ \phi_{\sigma}' \circ \sigma }$ acts on~${ V_p( E' ) }$. We
  thus find
  that~${ \rho_{A , \ideal{p}} \sim \beta(\sigma)^{-1} \, \rho_{E',
      p}(\sigma) }$ for all~${ \sigma \in G_{\Q} }$
  with~${ \phi_\sigma' = 1 }$.
\end{remark}

For each~${ \sigma \in G_{\Q} }$ we note that
\begin{IEEEeqnarray*}{rCl}
  \det \rho_{A, \ideal{p}}( \sigma )
  & = & \det \left( \beta(\sigma)^{-1} \, \phi_\sigma \circ \sigma \right)
  = \beta( \sigma )^{-2} \, \det \left( \phi_\sigma \circ \sigma \right) \\
  & = &\beta( \sigma )^{-2} \, d( \sigma ) \, \chi_p( \sigma )
  = \varepsilon( \sigma )^{-1} \, \chi_p( \sigma ),
\end{IEEEeqnarray*}
where~${ \varepsilon }$ is the splitting character of the splitting
map~${ \beta }$,~${ d }$ is the degree map and~${ \chi_p }$ is
the~${ p }$-cyclotomic character. The fact
that~${ \det \left( \phi_\sigma \circ \sigma \right) = d( \sigma ) \,
  \chi_p( \sigma ) }$ can be easily shown using the Weil
pairing.

Theorem~4.4 in \cite{Ribet} tells us that~${ A }$ is a quotient
of~${ J_1 ( N ) }$ for some~${ N > 0 }$. Looking at the proof we in
fact have that~${ A }$ is isogenous to the quotient~${ A_f }$
of~${ J_1(N) }$ associated to a particular weight 2 newform~${ f }$ of
level~${ N }$. The character of this newform is according to
Lemma~${ 4.3 }$ and Lemma~${ 3.1 }$ in \cite{Ribet} equal
to~${ \chi_p^{-1} \cdot \det \rho_{A, \ideal{p}} = \varepsilon^{-1}
}$. Furthermore~${ A }$ and~${ A_f }$ being isogenous implies
that~${ N^{ \dim A } }$ is equal to the conductor of~${ A }$.

Applying the above to all factors of~${ B = \End_{\Q}^{F} E }$ we find
that
\begin{equation*}
  B \text{ is } \Q \text{-isogenous to } \prod_{i} A_{f_i},
\end{equation*}
for each~${ f_i }$ a newform of weight~2, level~${ N_i }$ and
character~${ \varepsilon_i^{-1} }$. Here~${ \varepsilon_i }$ is the
splitting character corresponding to~${ A_{f_i} }$. Using the formula
to compute the conductor of the restriction of scalars
from~\cite[Proposition~1]{Milne} we now have two ways to compute the
conductor~${ N_B }$ of~${ B }$, hence
\begin{equation}
  \label{eq:conductorRelation}
  \Delta_F^2 \, \curl{N}_{\Q}^{F} N_E = N_B = \prod_{i} N_i^{\dim A_{f_i}},
\end{equation}
where~${ N_E }$ is the conductor of~${ E }$ over~${ F
}$,~${ \Delta_F }$ is the discriminant of~${ F }$,
and~${ \curl{N}_{\Q}^{F} }$ is the norm of~${ F }$.

Note that the~${ \ideal{p} }$-adic Galois
representation~${ \rho_{f_i, \ideal{p}} : G_\Q \to \GL_2( K_{f_i,
    \ideal{p} } )}$ is in fact
the same as~${ \rho_{A_{f_i}, \ideal{p} } : G_\Q \to \GL_2( K_{f_i,
    \ideal{p}} ) }$ by definition. Here~${ K_{f_i} }$ is the
coefficient field of~${ f_i }$ and we
have~${ K_{f_i} = \End A_{f_i} \otimes \Q }$. If~${ \beta_i }$ is the
splitting map corresponding to~${ A_{f_i} }$ then
each~${ \rho_{f_i, \ideal{p} } }$ is in fact defined
by~${ \beta_i^{-1} }$ times the same action on~${ V_p(E) }$ as
described before. This implies that for~${ i \ne j }$ we
have~${ \rho_{f_j, \ideal{p}} = \beta_i \beta_j^{-1} \rho_{f_i,
    \ideal{p}} }$, where~${ \chi = \beta_i \beta_j^{-1} }$ is a
character. Therefore the coefficients of~${ f_j }$ are a twist of the
coefficients of~${ f_i }$, hence~${ f_j }$ is a twist of~${ f_i }$
by~${ \chi }$.

\begin{remark}
  In the case there is only one factor~${ A_{f_i} }$
  equation~\eqref{eq:conductorRelation} is already sufficient to
  determine the level of the corresponding newform~${ f_i }$. This is
  the case in \cite{Ellenberg}, \cite{DieulefaitUrroz}, \cite{Chen},
  and \cite{BennettChen}. In \cite{DieulefaitFreitas} there are two
  factors, but this problem is solved by studying the relation between
  the conductors of the Galois
  representations~${ \rho_{A_{f_i}, \ideal{p}} }$ rather than the
  relation between the levels of the~${ f_i }$. In \cite{Chen1} there
  is also two factors and the same approach as the one here is used.
\end{remark}

In \cite{AtkinLi} Atkin and Li give results on how the level of a
newform can change under a twist. This allows us to relate the
levels~${ N_i }$ of the newforms related to~${ E }$. Combined with the
formula in \eqref{eq:conductorRelation} this gives a way to compute a
finite list of candidates for the levels. For completeness we
formulate the result by Atkin and Li we use here.
\begin{proposition}
  \label{thm:newformTwist}
  Let~${ f }$ be a newform of level~${ N }$, weight~${ k }$ and
  character~${ \varepsilon }$ and let~${ \chi }$ be a Dirichlet
  character of level~${ p^\beta }$ for some prime number~${ p
  }$. Let~${ \alpha }$ and~${ \gamma }$ be the valuation of~${ p }$ in
  the conductor of~${ \varepsilon }$ and~${ \varepsilon \chi }$
  respectively. Let~${ \delta = \order_p N }$
  and~${ \delta' = \max \{ \delta, \beta + 1, \beta + \gamma \} }$,
  then for the newform~${ g }$ that is~${ f }$ twisted by~${ \chi }$
  we have
  \begin{itemize}
  \item ${ g }$ is a cusp form of level~${ p^{\delta' - \delta} N }$,
    character~${ \varepsilon \chi^2 }$ and weight~${ k }$.
  \item ${ g }$ is not a cusp form of a level~${ N' }$ such
    that~${ \order_q( N' ) < \order_q( N ) }$ for any
    prime~${ q \ne p }$.
  \item ${ g }$ is not a cusp form of a level~${ N' }$
    with~${ \order_p( N' ) < \delta' }$ if either
    \begin{enumerate}
    \item ${ \delta > \max \{ \beta + 1, \beta + \gamma \} }$,
    \item ${ \delta < \max \{ \beta + 1, \beta + \gamma \} }$ and~${ \gamma \ge 2 }$, or
    \item ${ \alpha = \beta = \gamma = \delta = 1 }$.
    \end{enumerate}
  \end{itemize}
\end{proposition}
\begin{remark}
  Proposition~\ref{thm:newformTwist} is an analogue of Theorem~3.1 in
  \cite{AtkinLi}. The formulation is different and stronger as the
  reasoning given in \cite{AtkinLi} can be used to prove this stronger
  result. However the arguments in \cite{AtkinLi} fail for case 3 in
  the last part, which was noted in \cite{ShemanskeWalling}. A proof
  for this case was given by Li and can be found in
  \cite{ShemanskeWalling} for the case of Hilbert modular forms. This
  proof also applies to modular forms.
\end{remark}

We now apply this theory to our curves~${ E_1 }$ and~${ E_2 }$ to
obtain the following result.

\begin{theorem}
  \label{thm:QcurveModularity}
  For each~${ i \in \{ 1, 2 \} }$ there exists a factor~${ A_{i, j} }$
  such that~${ A_{i, j} }$ is~${ \Q }$-isogenous to the abelian
  variety~${ A_f }$ of a newform
  \begin{IEEEeqnarray*}{rCll}
    f & \in & S_2 \left( \Gamma_1 \left( 2^{9} \cdot 3^{2} \cdot 5 \, \rad_{30} c \right), \varepsilon
    \right) & \quad \text{if } i = 1 \text{, } b \text{ even,} \\
    f & \in & S_2 \left( \Gamma_1 \left( 2^{8} \cdot 3^{2} \cdot 5 \, \rad_{30} c \right), \varepsilon
    \right) & \quad \text{if } i = 1 \text{, } b \text{ odd,} \\
    f & \in & S_2 \left( \Gamma_1 \left( 2^{10} \cdot 3 \cdot 5 \, \rad_{30} c \right), \varepsilon
    \right) & \quad \text{if } i = 2.
  \end{IEEEeqnarray*}
  Here~${ \rad_{30} c }$ is the product of all primes~${ p \mid c }$
  with~${ p \nmid 30 }$ and~${ \varepsilon }$ is one of the two
  Dirichlet characters of conductor~${ 15 }$ and order~${ 4 }$ for
  which the choice does not matter.
\end{theorem}
\begin{proof}
  By the theory above each factor~${ A_{i, j} }$ is isogenous to some
  abelian variety~${ A_f }$ of some newform~${ f }$. The character of
  these newforms is by the same theory equal to the inverse of a
  corresponding splitting character. The code~\cite{code} computes
  such a splitting character for each~${ A_{i, j} }$, which are all
  one of the two characters mentioned. Since the mentioned characters
  are Galois conjugates of each other, and since Galois conjugates of
  splitting maps correspond to the same factor~${ A_{i, j} }$ the
  choice indeed does not matter.

  For the levels of these newforms we first compute the conductors of
  the curves~${ E_{i, \gamma} }$ over~${ K_{\beta} }$. As explained in
  the proof of Proposition~\ref{thm:conductorK} the code~\cite{code}
  can calculate these conductors to be
  \begin{equation*}
    \curl{N}_i =
    \begin{cases}
      \left( 2^{6} \cdot 3 \, \rad_{30} c \right) &
      \text{if } i = 1 \text{ and } 2 \mid b \\
      \left( 2^{5} \cdot 3 \, \rad_{30} c \right) &
      \text{if } i = 1 \text{ and } 2 \nmid b \\
      \left( 2^{7} \, \rad_{30} c \right) & \text{if } i = 2,
    \end{cases}
  \end{equation*}
  where~${ \curl{N}_i }$ is the conductor of~${ E_{i, \gamma} }$.
  From this we can compute the left-hand side of
  equation~\eqref{eq:conductorRelation} to find that
  \begin{equation*}
    N_i =
    \begin{cases}
      2^{72} \cdot 3^{16} \cdot 5^{12} \left( \rad_{30} c \right)^8 &
      \text{if } i = 1 \text{ and } 2 \mid b \\
      2^{64} \cdot 3^{16} \cdot 5^{12} \left( \rad_{30} c \right)^8 &
      \text{if } i = 1 \text{ and } 2 \nmid b \\
      2^{80} \cdot 3^{8} \cdot 5^{12} \left( \rad_{30} c \right)^8 & \text{if } i = 2,
    \end{cases}
  \end{equation*}
  where~${ N_i }$ is the conductor
  of~${ \residue_{\Q}^{K_{\beta}} E_{i, \gamma} }$. The conductor of
  the elliptic curves was computed using a version of Tate's algorithm
  implemented in the code~\cite{code} that works for Frey curves.

  For each~${ i }$ the newforms~${ f_{i,1} }$ and~${ f_{i,2} }$
  corresponding to~${ A_{i, 1} }$ and~${ A_{i, 2} }$ are twists of one
  another by the character~${ \chi = \epsilon_8 \epsilon_5 }$ and its
  inverse. Here~${ \epsilon_8 }$ is the character of conductor~${ 8 }$
  with~${ \epsilon_8(-1) = -1 }$ and~${ \epsilon_5 }$ is a character
  with conductor~${ 5 }$ and order~${ 4 }$. We can thus apply
  Proposition~\ref{thm:newformTwist} by first twisting
  with~${ \epsilon_8 }$ and then with~${ \epsilon_5 }$
  or~${ \epsilon_5^{-1} }$. We immediately see that the levels should
  be the same for all primes~${ p \ne 2, 5 }$.

  Note that for the order of~${ 2 }$ in the levels~${ N_{i,1} }$
  and~${ N_{i,2} }$ of~${ f_{i,1} }$ and~${ f_{i,2} }$ respectively we know that
  \begin{equation*}
    4 \order_2 N_{i,1} + 4 \order_2 N_{i,2} \ge 64,
  \end{equation*}
  by equation~\eqref{eq:conductorRelation}, hence the order of~${ 2 }$
  in one of the two is at least~${ 8 }$. Since all splitting
  characters and twist characters can be defined modulo~${ 120 }$
  the~${ \beta }$ and~${ \gamma }$ in
  Proposition~\ref{thm:newformTwist} can never
  exceed~${ \order_2 120 = 3 }$. This implies that we are in case~1 in
  the last list of Proposition~\ref{thm:newformTwist} and the order
  of~${ 2 }$ in both levels must be the same.

  Now for the order of~${ 5 }$ we have
  \begin{equation*}
    4 \order_5 N_{i,1} + 4 \order_5 N_{i,2} = 12.
  \end{equation*}
  Since the characters of the corresponding newforms have a conductor
  divisible by~${ 5 }$ at least one factor~${ 5 }$ has to appear in
  both levels. This implies that one of the levels has a single
  factor~${ 5 }$ and the other has a factor~${ 5^2 }$. By picking the
  first we get the result as stated in this theorem.
\end{proof}
\begin{remark}
  The reasoning done in the proof to determine the order of a
  prime~${ p }$ in the level of a newform can be generalized to an
  arbitrary case, however it might result in multiple candidates for
  levels. The code~\cite{code} automates this procedure to compute the
  part of the level consisting of bad primes, i.e. primes which ramify
  in the decomposition field, divide the conductor of one of the
  relevant characters, or are divisible by a prime of additive
  reduction for the elliptic curve.
\end{remark}


\subsection{Level lowering}
\label{sec:LL}

The levels appearing in Theorem~\ref{thm:QcurveModularity} still
depend on the specific solution~${ (a, b, c) }$ of
equation~\eqref{eq:main}. To get rid of the additional primes we will
need to look at the mod~${ l }$ Galois representations and use some
level lowering results.

Let us first introduce what we consider to be the mod~${ p }$ Galois
representation. For an elliptic curve~${ E }$ defined over~${ K }$ and
a prime number~${ p }$, the mod~${ p }$ Galois representation is
simply
\begin{equation*}
  \overline{\rho}_{E, p} : G_K \to \End E[p] \cong \GL_2( \F_p ),
\end{equation*}
induced by the Galois action on~${ p }$-torsion points. This naturally
is the reduction of the~${ p }$-adic Galois representation
\begin{equation*}
  \rho_{E, p} : G_K \to \End V_p(E) \cong \GL_2( \Q_p ),
\end{equation*}
induced by the Galois action on the Tate module~${ V_p( E ) }$. For
a~${ \Q }$-simple abelian variety~${ A }$ of~${ \GL_2 }$-type
with~${ F = \End A \otimes \Q }$, we have seen that the Galois action
on the Tate module~${ V_p( A ) }$ factors into
multiple~${ \ideal{p} }$-adic Galois
representations~${ \rho_{A, \ideal{p}} : G_\Q \to \GL_2( F_{\ideal{p}}
  ) }$, as~${ V_p( A ) }$ is a 2-dimensional vector space over the
algebra~${ \End A \otimes \Q_p = \prod_{\ideal{p} \mid p} F_{
    \ideal{p}} }$. Similarly the action of~${ G_\Q }$ on
the~${ p }$-torsion points factors into multiple mod~${ \ideal{p} }$
Galois representations
\begin{equation*}
  \overline{\rho}_{A, \ideal{p}} : G_\Q \to \GL_2( \F_{\ideal{p}} ),
\end{equation*}
as~${ A[l] }$ is a~${ 2 }$-dimensional vector space
over~${ \End A \otimes \F_p = \prod_{\ideal{p} \mid p} \F_{\ideal{p}}
}$. Again these can be considered as reductions of
their~${ \ideal{p} }$-adic counterparts. For newforms the
mod~${ \ideal{p} }$ Galois representations are defined as those of the
corresponding abelian variety.

To apply level lowering results we first need to know that the
mod~${ l }$ Galois representations of~${ E_1 }$ and~${ E_2 }$ are
absolutely irreducible and unramified at the primes that should be
eliminated from the level.

\begin{theorem}
  \label{thm:irreducible}
  The mod~${ l }$ Galois
  representation~${ \overline{\rho}_{E_i, l} : G_{K} \to \End E_i [l]
    \cong \GL_2( \F_l ) }$ is irreducible for any~${ i = 1, 2 }$
  and~${ l > 5 }$.
\end{theorem}
\begin{proof}
  Suppose that the mod~${ l }$ Galois representation of one of these
  curves is reducible for some prime~${ l > 5 }$. If~${ l = 11 }$
  or~${ l > 13}$ Proposition~3.2 of \cite{Ellenberg} tells us that the
  corresponding curve has potentially good reduction at all primes of
  characteristic~${ p > 3 }$. This would imply that
  the~${ j }$-invariant is integral at all those primes
  (\cite[VII.5.5]{Silverman1}), which contradicts
  Lemma~\ref{thm:j_integral}. Therefore~${ l \in \{ 7, 13 \} }$.

  Now note that the mod~${ l }$ representation being reducible means
  that the corresponding curve has an~${ l }$-isogeny. Both~${ E_1 }$
  and~${ E_2 }$ also have a~${ 2 }$-isogeny and are defined
  over~${ K }$, hence the curve must correspond to a~${ K }$-point on
  the curve~${ X_0( 2 l ) }$. They are however not~${ \Q }$-points as
  we have seen that the minimal field over which a curve in their
  isogeny class can be defined is~${ K }$.
  We do however know that both curves are 2-isogenous to their Galois
  conjugate, hence would correspond to a~${ \Q }$-point of the
  quotient~${ C_l = X_0( 2 l ) / w_2 }$, where~${ w_2 }$ is the
  Atkin-Lehner-involution. We study such points using
  MAGMA~\cite{magma} for the remaining cases~${ l = 7, 13 }$.

  In the cases~${ l = 7, 13 }$ the curve~${ C_l }$ is an elliptic
  curve with only finitely many rational points. Note that these
  points correspond to~${ \Q( \sqrt{-7} ) }$ points
  on~${ X_0 ( 2 l ) }$ for~${ l = 7 }$ and to~${ \Q( \sqrt{13} ) }$
  points on~${ X_0 ( 2 l ) }$ for~${ l = 13 }$. Since~${ E_1 }$
  and~${ E_2 }$ do not correspond to rational points this gives a
  contradiction.
\end{proof}
\begin{remark}
  Note that the proof above can also be extended to the
  cases~${ l = 3, 5 }$. In this case~${ X_0 ( 2 l ) }$ is a genus 0
  curve, but by explicitly writing out the quotient
  map~${ X_0 (2 l ) \to C_l }$ we can see that no~${ K }$-points
  on~${ X_0 (2 l ) }$ map to~${ \Q }$-points on~${ C_l }$. Since we
  will not need these cases, the proof has been left out.
\end{remark}
\begin{corollary}
  \label{thm:irreducibleNewform}
  Let~${ f }$ be a newform as in Theorem~\ref{thm:QcurveModularity}
  and~${ l > 5 }$ be a prime number, then for each
  prime~${ \lambda \mid l }$ in the coefficient field of~${ f }$ the
  mod~${ \lambda }$ Galois
  representation~${ \overline{\rho}_{f, \lambda} : G_\Q \to
    \GL_2(\F_{\lambda}) }$ is absolutely irreducible.
\end{corollary}
\begin{proof}
  First of all we note
  that~${ \overline{\rho}_{f, \lambda} \sim \overline{\rho}_{A_{i, j},
      \lambda} }$ for the corresponding factor~${ A_{i, j} }$ in
  Theorem~\ref{thm:QcurveModularity}. By the discussion on
  page~\pageref{sec:galoisDiscussion} we know
  that~${ \overline{\rho}_{A_{i,j}, \lambda}(\sigma) }$ is completely
  determined by the action of~${ \beta( \sigma ) \circ \sigma }$
  on~${ A_{i, j} }$. This is the same as the action
  of~${ \phi_{\sigma, \gamma} \circ \sigma }$ on~${ E_{i, \gamma} }$.
  As remarked on page~\pageref{sec:galoisDiscussion} this is by
  construction the same as the action
  of~${ \phi_\sigma \circ \sigma }$ on~${ E_i }$.
  Since~${ \phi_{\sigma} = 1 }$ for all~${ \sigma \in G_K }$, we have
  that
  \begin{equation*}
    \overline{\rho}_{A_{i,j}, \lambda} |_{G_K} =
    \beta^{-1} |_{G_K} \cdot \overline{\rho}_{E_i, l}.
  \end{equation*}
  Since~${ \overline{\rho}_{E_i, l} }$ is irreducible by
  Theorem~\ref{thm:irreducible}, this implies
  that~${ \overline{\rho}_{A_{i,j}, \lambda} }$ is irreducible.

  To get
  that~${ \overline{\rho}_{f, \lambda} \sim \overline{\rho}_{A_{i,j},
      \lambda} }$ is absolutely irreducible, we note that this
  representation is odd by Lemma~3.2 of \cite{Ribet} and that~${ l }$
  is odd, in which case irreducibility and absolute irreducibility are
  equivalent.
\end{proof}

\begin{proposition}
  \label{thm:unramified}
  Let~${ f }$ be a newform as in Theorem~\ref{thm:QcurveModularity}
  and~${ l > 5 }$ be a prime number, then for each
  prime~${ \lambda \mid l }$ in the coefficient field of~${ f }$ the
  mod~${ \lambda }$ Galois
  representation~${ \overline{\rho}_{f, \lambda} : G_\Q \to \GL_2(
    \F_\lambda ) }$ is finite outside primes dividing~${ 30 }$. In
  particular it is unramified outside primes dividing~${ 30 \, l }$.
\end{proposition}
\begin{proof}
  As in Corollary~\ref{thm:irreducibleNewform} we know that
  \begin{equation*}
    \overline{\rho}_{f, \lambda} |_{G_K}
    \sim \beta^{-1} |_{G_K} \cdot \overline{\rho}_{E_i, l}
  \end{equation*}
  for some~${ i \in \{ 1, 2 \} }$. Since~${ \beta }$ is trivial
  on~${ G_{K_\beta} }$ we find that
  \begin{equation*}
    \overline{\rho}_{f, \lambda} |_{G_{K_\beta}}
    \sim \overline{\rho}_{E_i, l} |_{G_{K_\beta}}.
  \end{equation*}
  Since the discriminant of~${ K_\beta }$ is only divisible by the
  prime numbers~${ 2 }$,~${ 3 }$ and~${ 5 }$, we find that the
  ramification subgroup~${ I_p }$ of a prime number~${ p \nmid 30 }$
  is contained in~${ G_{K_\beta}}$. Note
  that~${ \overline{\rho}_{f , \lambda} }$ being finite
  at~${ p \nmid 30 }$ only depends
  on~${ \overline{\rho}_{f , \lambda} |_{I_p} \sim
    \overline{\rho}_{E_i, l} |_{I_p} }$, for which this was already
  proven in Proposition~\ref{thm:unramifiedOutside30}.
\end{proof}

We can now use level lowering results proven by Diamond in
\cite{DiamondLL} based on work by Ribet \cite{RibetLL} to lower the
level to something independent of the chosen solution~${ (a, b, c) }$.

\begin{theorem}
  \label{thm:rho_modular}
  For each elliptic curve~${ E_{i, \gamma} }$ and prime
  number~${ l > 5 }$ there exists a factor~${ A_{i, j} }$ as in
  Proposition~\ref{thm:QcurveModularity} such that for each prime
  ideal~${ \lambda \mid l }$
  of the field~${ \End A_{i, j} \otimes \Q = \Q(\zeta_8) }$ we
  have~${ \overline{\rho}_{ A_{i, j}, \lambda } \sim
    \overline{\rho}_{g, \lambda'} }$ for some prime
  ideal~${ \lambda' \mid l }$ in the appropriate field and a
  newform~${ g }$ satisfying
  \begin{IEEEeqnarray*}{rCll}
    g & \in & S_2 \left( \Gamma_1 \left( 23040 \right), \varepsilon
    \right) & \quad \text{if } i = 1 \text{, } b \text{ even,} \\
    g & \in & S_2 \left( \Gamma_1 \left( 11520 \right), \varepsilon
    \right) & \quad \text{if } i = 1 \text{, } b \text{ odd,} \\
    g & \in & S_2 \left( \Gamma_1 \left( 15360 \right), \varepsilon
    \right) & \quad \text{if } i = 2.
  \end{IEEEeqnarray*}
  Here~${ \varepsilon }$ is one of the two Dirichlet characters of
  conductor~${ 15 }$ and order~${ 4 }$. The choice does not matter.
\end{theorem}
\begin{proof}
  We start by picking some~${ f }$ as in
  Theorem~\ref{thm:QcurveModularity}. Let~${ A_{i, j} }$ be the
  corresponding factor. We already have that
  that~${ \overline{\rho}_{ A_{i, j}, \lambda } \sim
    \overline{\rho}_{f, \lambda} }$ for an arbitrary
  prime~${ \lambda \mid l}$. We will show that we can find a newform
  of the level as in this theorem which still has an isomorphic Galois
  representation.

  Note that~${ \overline{\rho}_{f, \lambda} }$ is irreducible by
  Corollary~\ref{thm:irreducibleNewform} and odd as it is the Galois
  representation of a newform.

  We apply Theorem~4.1 in \cite{DiamondLL}, which tells us we can find
  a newform~${ g }$ of weight 2 with an isomorphic Galois
  representation~${ \overline{\rho}_{g, \lambda} }$. The level
  of~${ g }$ is the level of~${ f }$ divided by all prime
  numbers~${ p }$ that appear only once in the level, do not
  divide~${ l }$ or the conductor of the character of~${ f }$, and at
  which the Galois representation is unramified. The levels and the
  explicit character in Proposition~\ref{thm:QcurveModularity} and the
  result from Proposition~\ref{thm:unramified} tell us that all those
  prime numbers~${ p }$ not dividing~${ 30 l }$ satisfy these
  conditions and can thus be removed from the level.

  Lastly we use Theorem~2.1 in \cite{RibetLL3} which shows that the
  same result holds for a newform~${ g }$ of weight 2 and a level in
  which all powers of~${ l }$ are also removed. The weight
  remains~${ 2 }$ as the Galois representation is finite at~${ l }$ by
  Proposition~\ref{thm:unramified} and the fact that~${ l > 5 }$. The
  resulting level is the one given in this theorem.
\end{proof}
\begin{remark}
  Note that the newforms in Theorem~\ref{thm:rho_modular} are in fact
  those that have the Serre level and weight of the corresponding
  irreducible representation~${ \overline{\rho}_{A_{i,j}, \lambda}
  }$. The result of this theorem would therefore also directly follow
  from the Serre conjectures.
\end{remark}


\subsection{Newform elimination}
\label{sec:elimination}

The strategy to complete the proof of Theorem~\ref{thm:main} is to
show that the conclusion of Theorem~\ref{thm:rho_modular} will give a
contradiction, implying that the implicit assumption of a
solution~${ (a, b, c) }$ to equation~\eqref{eq:main} existing
with~${ l > 5 }$ and~${ \gcd(a, b) = 1 }$ must be false. We derive
this contradiction by comparing traces of Frobenius
of~${ \overline{\rho}_{ E_{i, \gamma} , l} : G_{K_\beta} \to
  \GL_2(\F_l) }$
and~${ \overline{\rho}_{g, \lambda' } : G_\Q \to \GL_2(\F_\lambda) }$
for~${ g }$ in one of the given spaces. Note that since both
representations are defined over different Galois groups, we need a
small result.

\begin{lemma}
  \label{thm:frob_relations}
  We have
  that~${ \overline{\rho}_{ g, \lambda' } \left( \frob_{\ideal{p}}
    \right) \sim \overline{\rho}_{ g, \lambda' } \left( \frob_p^d
    \right) }$ for any prime~${ \ideal{p} }$ of~${ K_{\beta} }$ of
  characteristic~${ p \nmid 30 \, l }$ and a residue field of
  degree~${ d }$. Here~${ \sim }$ denotes the two are conjugates.
\end{lemma}
\begin{proof}
  Let~${ \ideal{ p } }$ be an arbitrary prime of~${ K_{\beta} }$ of
  characteristic~${ p \nmid 30 l }$. Note that a Frobenius
  element~${ \frob_{ \ideal{p} } \in G_{K_{\text{dec}, \ideal{p}}} }$
  maps to the homomorphism~${ x \mapsto x^{\# \F_{\ideal{p}} } }$
  inside~${ G_{\F_p} }$, just as does~${ \frob_p^d }$ for a Frobenius
  element~${ \frob_p \in G_{\Q_p}}$. This means their difference lies
  in the ramification subgroup of~${ G_{\Q_p} }$.
  Since~${ \overline{\rho}_{g, \lambda'} }$ is unramified at~${ p }$
  by Proposition~\ref{thm:unramified} we find
  that~${ \overline{\rho}_{ g, \lambda' } \left( \frob_{\ideal{p}}
    \right) \sim \overline{\rho}_{ g, \lambda' } \left( \frob_p^d
    \right) }$.
\end{proof}

Now the rest of the proof becomes a computation.

First we compute the newforms in the spaces mentioned in
Proposition~\ref{thm:rho_modular}. These computations take quite some
time, especially the computation for the newforms of
level~${ 15360 }$, which took approximately 5 days of computation time
in MAGMA~\cite{magma} using a desktop computer (Intel core i5-6600
CPU, 3.3 GHz). For comparison computing the space of newforms of
level~${ 11520 }$ took just under 8 minutes on the same machine and
the space of newforms of level~${ 23040 }$ took just over an hour. For
this reason all newforms were pre-computed and then stored by saving
the Fourier coefficients for all primes smaller than 500, as this data
is sufficient to compute the sought traces of Frobenius for those
primes.

The table below gives some general data about each space of
newforms. It lists from left to right the level of the newforms, the
dimension of the corresponding newspace, the number of Galois
conjugacy classes of newforms, the possible sizes of the Galois
conjugacy classes, and the total number of newforms among all
conjugacy classes. Note that the last is always twice the dimension
mentioned before, since the Galois conjugacy class of the character
consists of two characters.

\begin{table}[h]
  \centering
  \begin{tabular}{c|c|c|c|c}
    level & dim. & \# conj. classes & size of conj. classes & \# newforms \\
    \hline
    ${ 11520 }$ & ${ 192 }$ & ${ 30 }$ & ${ 4, 8, 16, 24, 32, 48 }$ & ${ 384 }$ \\
    ${ 23040 }$ & ${ 384 }$ & ${ 20 }$ & ${ 8, 40, 48 }$ & ${ 768 }$ \\
    ${ 15360 }$ & ${ 752 }$ & ${ 14 }$ & ${ 16, 64, 80, 96, 128, 176, 192 }$ & ${ 1504 }$
  \end{tabular}
  \caption{Data of the computed newforms}
  \label{tab:newformData}
\end{table}

For some small primes~${ \ideal{p} \nmid 30 l }$ of~${ K_{\beta} }$ we
compute the set of possible values
of~${ \tr \overline{\rho}_{E_{i, \gamma}, l}( \frob_{\ideal{p}} )
}$. A standard result is that
\begin{equation*}
  \tr \overline{\rho}_{E_{i, \gamma}, l}( \frob_{\ideal{p}} ) =
  \begin{cases}
    \# \F_{\ideal{p}} + 1 - \# E_{i, \gamma}( \F_{\ideal{p}} ) \text{ mod } l & \\
    \# \F_{\ideal{p}} + 1 \text{ mod } l & \\
    - \# \F_{\ideal{p}} - 1 \text{ mod } l &
  \end{cases}
\end{equation*}
where the cases correspond to~${ E }$ having good, split
multiplicative and non-split multiplicative reduction
at~${ \ideal{ p } }$ respectively. Note that the right hand side is
each time the reduction of an integer that does not depend on~${ l }$,
but does depend on the chosen solution~${ (a, b, c) }$ used to
form~${ E_{i, \gamma} }$. In fact it only depends on the value
of~${ a }$ and~${ b }$ modulo~${ p }$, the characteristic
of~${ \ideal{p} }$. Ignoring the
case~${ a \equiv b \equiv 0 }$~(mod~${ p }$) as~${ a }$ and~${ b }$
are coprime, we compute the set~${ S_{i, \ideal{p}} }$ of possible
integers on the right hand side from all possible values of~${ a }$
and~${ b }$ modulo~${ p }$. The reduction type in these cases can be
computed using the version of Tate's algorithm implemented in the
code~\cite{code}.

On the other hand we compute the values
of~${ \tr \overline{\rho}_{g, \lambda'}( \frob_p^d ) }$ for each
newform~${ g }$ found before, where~${ p }$ is the characteristic
of~${ \ideal{p} }$,~${ d = [ \F_{\ideal{p}} : \F_p ] }$,
and~${ \lambda' \mid l }$ is the prime ideal corresponding to a
fixed~${ \lambda \mid l }$ in Proposition~\ref{thm:rho_modular}. This
trace can be computed
from~${ \tr \overline{\rho}_{g, \lambda'}( \frob_p ) }$
and~${ \det \overline{\rho}_{g, \lambda'}( \frob_p ) }$ by the fact
that for a 2-by-2 matrix~${ A }$ the value of~${ \tr A^d }$ can be
expressed as a polynomial in~${ \tr A }$ and~${ \det A }$.
Since~${ p }$ does not divide the level we have
\begin{IEEEeqnarray*}{rCl}
  \tr \overline{\rho}_{g, \lambda'}( \frob_p ) & = & a_p( g ) \text{ mod } \lambda' \\
  \det \overline{\rho}_{g, \lambda'}( \frob_p ) & = & \varepsilon(p) p \text{ mod } \lambda',
\end{IEEEeqnarray*}
where~${ a_p(g) }$ is the~${ p }$-th coefficient in the Fourier
expansion of~${ g }$ and~${ \varepsilon }$ is the character
of~${ g }$. Note again that the right hand side for both these values
is the reduction of an algebraic integer that is independent
of~${ \lambda' }$. Using these algebraic integers in the formula
for~${ \tr A^d }$ we get an algebraic integer~${ t_{g, \ideal{p}} }$
independent of~${ \lambda' }$ of which the reduction mod~${ \lambda' }$
is the value of~${ \tr \overline{\rho}_{g, \lambda'}( \frob_p^d ) }$.

Now for each newform~${ g }$ with corresponding
curve~${ E_{i, \gamma} }$ and a prime ideal~${ \ideal{p} \nmid 30 }$ we
can compute the integer
\begin{equation*}
  M_{ g, \ideal{p} } = p \, \prod_{s \in S_{i, \ideal{p}}} \curl{N} (s - t_{g, \ideal{p}}),
\end{equation*}
where~${ \curl{N} }$ denotes the norm of the appropriate number field.
By Theorem~\ref{thm:rho_modular} and Lemma~\ref{thm:frob_relations}
this algebraic integer is divisible by~${ l }$ since
either~${ l = p }$ or one of the factors~${ s - t_{g, \ideal{p}} }$
reduces to
\begin{equation*}
  \tr \overline{\rho}_{E_i, l}( \frob_{\ideal{p}} ) - \tr
  \overline{\rho}_{g, \lambda'} ( \frob_p^d ) = 0
\end{equation*}
modulo~${ \lambda' }$. Therefore if~${ g }$ would be the newform
obtained in Theorem~\ref{thm:rho_modular} corresponding
to~${ E_{i, \gamma} }$, then~${ l > 5 }$ should divide the norm
of~${ M_{\ideal{p}} }$.

We now apply this result. For each newform~${ g }$ in the spaces of
Theorem~\ref{thm:rho_modular} we compute the integer
\begin{equation*}
  M_g = \gcd \{ M_{g, \ideal{p}} : \ideal{p} \nmid 30 \text{ of characteristic } p < 40 \}.
\end{equation*}
Furthermore we remove all factors~${ 2 }$,~${ 3 }$ and~${ 5 }$
from~${ M_g }$, giving us a number divisible only by all~${ l > 5 }$
for which the newform~${ g }$ can still be the one in
Theorem~\ref{thm:rho_modular}. We eliminate all newforms for
which~${ M_g = 1 }$ which leaves us with 14 newforms of level 11520,
12 newforms of level 23040 and 7 newforms of level 15360.

The last step is to use both Frey curves simultaneously. This is known
as a multi-Frey approach and was also used in
\cite{DieulefaitFreitas}, \cite{BennettChen} and \cite{Chen}. Instead
of computing the sets~${ S_{1, \ideal{p}} }$
and~${ S_{2, \ideal{p} } }$ independently we now compute one
set~${ S_{\ideal{p}} \subset \Z^2}$. This set contains for each
congruence class of~${ (a, b) }$ modulo~${ p }$ a
pair~${ (s_1, s_2) }$ such that~${ s_i }$ is the integer that reduces
to~${ \tr \overline{\rho}_{E_{i, \gamma}, l}( \frob_{\ideal{p}} ) }$
for that congruence class.

Now for each prime ideal~${ \ideal{p} \nmid 30 }$ and each pair of
newforms~${ (g_1, g_2) }$ remaining with~${ g_i }$ corresponding
to~${ E_{i, \gamma} }$ we can compute
\begin{equation*}
  M_{g_1, g_2, \ideal{p}} = p \prod_{(s_1, s_2) \in S_{\ideal{p}}}
  \gcd(\curl{N} (s_1 - t_{g_1, \ideal{p}}), \curl{N} (s_2 - t_{g_2, \ideal{p}})),
\end{equation*}
where~${ \curl{N} }$ again denotes the norm in the appropriate
field. As before this integer is divisible by all~${ l > 5 }$ for
which~${ g_1 }$ and~${ g_2 }$ are newforms for which
Theorem~\ref{thm:rho_modular} holds.

We compute
\begin{equation*}
  M_{g_1, g_2} = \gcd \{ \curl{N} M_{g_1, g_2, \ideal{p}} :
  \ideal{p} \nmid 30 \text{ of characteristic } p < 50 \},
\end{equation*}
for each pair of newforms~${ (g_1, g_2) }$ and see that none of them
are divisible by primes~${ l > 5 }$. If a solution~${ (a, b, c) }$
with~${ \gcd(a, b) = 1 }$ to equation~\eqref{eq:main} would exist
for~${ l > 5 }$, then this would contradict
Theorem~\ref{thm:rho_modular}. Therefore no such solution to
equation~\eqref{eq:main} can exist, proving Theorem~\ref{thm:main}
for~${ l > 5 }$ prime.

\begin{remark}
  Most prime exponents~${ l > 5 }$ can already be eliminated by only
  looking at the curve~${ E_{ 2, \gamma } }$ at more primes then
  considered here and using more restrictions on~${ a }$ and~${ b
  }$. However it seems impossible to eliminate the case~${ l = 7 }$ in
  this way, hence the use of the multi-Frey curve approach.
\end{remark}



\bibliographystyle{plain}
\bibliography{previous,technical,code}

\end{document}